\documentclass[11pt,twoside,reqno]{amsart}

\usepackage{amsmath,amsfonts,amsthm,amssymb}
\usepackage{graphicx,psfrag,overpic}

\usepackage{bm,enumerate}
\usepackage{cases}
\usepackage[all,cmtip,line]{xy}
\usepackage{array,CJK}
\usepackage{array,setspace}
\usepackage{color}
\usepackage{yhmath}
\usepackage{appendix}

\usepackage[utf8]{inputenc}
\usepackage[T1]{fontenc}

\numberwithin{equation}{section}
\numberwithin{figure}{section}

\renewcommand{\thefootnote}{\ensuremath{\fnsymbol{footnote}}}

\usepackage{graphics}
\usepackage{epsfig}
\newtheorem{claim}{\bf \t}[part]

\newtheorem{theorem}{Theorem}[section]

\newtheorem{lemma}{Lemma}[section]

\newtheorem{remark}{Remark}[section]
\newtheorem{definition}{Definition}[section]

\newcommand{\qnt}[1]{\left(#1\right)}

\title[Weak-Strong Uniqueness for a RigiD Body in an Inviscid Compressible Fluid]{Weak-Strong Uniqueness for a Rigid Body Immersed in an Inviscid Compressible Fluid}

\author{Qianfeng Li}
\address{Department of Mathematics,
Friedrich-Alexander-Universität Erlangen-Nürnberg, Cauerstr. 11, 91058, Germany }
\email{qfli1995@gmail.com}

\author{Emil Wiedemann$^{*}$}
\address{Department of Mathematics,
Friedrich-Alexander-Universität Erlangen-Nürnberg, Cauerstr. 11, 91058, Germany }
\email{emil.wiedemann@fau.de}

\begin{document}
	\begin{abstract}
		\,\,We consider the coupled motion of a free rigid body immersed in an inviscid compressible isentropic fluid. By means of a vanishing viscosity limit, we obtain the local-in-time existence of a dissipative measure-valued solution to the model. Moreover, we establish the weak-strong uniqueness property of the obtained measure-valued solution. To our knowledge, this is the first mathematical result on compressible inviscid fluid-structure interaction. The key novel technique is the construction of a suitable approximation of the test function in the weak formulation of the inviscid system, as the space of test functions depends on the viscosity parameter.  
	\end{abstract}
	
	\keywords{fluid-structure interaction model, compressible Euler equations, dissipative measure-valued solution, weak-strong uniqueness, vanishing viscosity limit}
	\maketitle

\begingroup
\renewcommand\thefootnote{\fnsymbol{footnote}}
\footnotetext[1]{Corresponding author.}
\endgroup

	\section{Introduction}
\subsection{Mathematical description}

We consider the coupled motion of an isentropic, compressible, inviscid fluid and a (possibly mass-inhomogeneous) rigid body in a fixed domain $\Omega \subset \mathbb{R}^3$. At time $t \in \mathbb{R}^+$, the rigid body occupies $\mathcal{B}_t \subset \Omega$ and the fluid occupies $\mathcal{F}_t := \Omega \setminus \mathcal{B}_t$.
The fluid density, velocity, and pressure are denoted by $\rho_F$, $u_F$, and $p_F$, respectively, while $\rho_B$ and $u_B$ denote the density and velocity of the rigid body.

Since the solid is rigid, its velocity field decomposes into translation and rotation:
\begin{equation}\label{eq:rigid-vel}
    u_B(t,x) = V(t) + \omega(t) \times \bigl(x - X(t)\bigr),
\end{equation}
where $V(t)$ and $\omega(t)$ are the translational and angular velocities, and $X(t)$ is the center of mass,
\begin{equation}\label{eq:com}
    M := \int_{\mathcal{B}_t} \rho_B(t,x)\, \mathrm{d}x,
    \qquad
    X(t) := \frac{1}{M}\int_{\mathcal{B}_t} \rho_B(t,x)\, x \,\mathrm{d}x .
\end{equation}
For an isentropic polytropic fluid,
\begin{equation}\label{eq:state}
    p_F = \rho_F^{\gamma}, \qquad \gamma>1.
\end{equation}

The coupled dynamics read
\begin{equation}\label{eq:model}
\left\{
\begin{aligned}
    &\partial_t \rho_F + \operatorname{div}(\rho_F u_F) = 0,
    && t>0,\; x \in \mathcal{F}_t, \\[2pt]
    &\partial_t(\rho_F u_F) + \operatorname{div}\bigl(\rho_F u_F \otimes u_F + p_F \mathbb{I}\bigr) = 0,
    && t>0,\; x \in \mathcal{F}_t, \\[2pt]
    &\partial_t \rho_B + u_B \cdot \nabla \rho_B = 0,
    && t>0,\; x \in \mathcal{B}_t, \\[2pt]
    &X'(t) = V(t), \\[2pt]
    &M V'(t) = \displaystyle\int_{\partial \mathcal{B}_t} p_F\, n \, \mathrm{d}S, \\[6pt]
    &\bigl(\mathcal{J}(t)\,\omega(t)\bigr)' 
        = \displaystyle\int_{\partial \mathcal{B}_t} \bigl(x - X(t)\bigr) \times \bigl(p_F\, n\bigr)\, \mathrm{d}S,
\end{aligned}
\right.
\end{equation}
with slip boundary conditions
\begin{equation}\label{eq:bc}
\left\{
\begin{aligned}
    &u_F \cdot n = 0, && t>0,\; x \in \partial\Omega,\\
    &(u_F - u_B) \cdot n = 0, && t>0,\; x \in \partial\mathcal{B}_t,
\end{aligned}
\right.
\end{equation}
and initial data
\begin{equation}\label{eq:ic}
\left\{
\begin{aligned}
    &(\rho_F, u_F)\big|_{t=0} = \bigl(\rho_F^0(x),\, u_F^0(x)\bigr),
        && x \in \mathcal{F}_0,\\
    &(X,V,\omega)\big|_{t=0} = \bigl(X_0,\, V_0,\, \omega_0\bigr),\\
    &(\rho_B, u_B)\big|_{t=0} = \bigl(\rho_B^0(x),\, V_0 + \omega_0 \times (x-X_0)\bigr),
        && x \in \mathcal{B}_0.
\end{aligned}
\right.
\end{equation}
Here, and throughout this paper, $n$ is the inward unit normal to $\partial\mathcal{F}_t$ (hence on the fluid-solid interface the \emph{outward} unit normal on $\partial\mathcal{B}_t$), and $\mathbb{I}$ is the $3\times 3$ identity matrix. The inertia tensor is
\begin{equation}\label{eq:inertia}
    \mathcal{J}(t) = \int_{\mathcal{B}_t} \rho_B(t,x)\,\Bigl( |x - X(t)|^2 \mathbb{I}
    - (x - X(t)) \otimes (x - X(t)) \Bigr)\, \mathrm{d}x .
\end{equation}

The solid configuration can be represented as
\begin{equation}\label{eq:body-domain}
    \mathcal{B}_t = \bigl\{\, X(t) + \mathbb{O}_t\, x : x \in \mathcal{B}_0 \,\bigr\},
\end{equation}
where $\mathbb{O}_t \in SO(3)$ is the rotation matrix. Its evolution satisfies
\begin{equation}\label{eq:rotation-ode}
    \dot{\mathbb{O}}_t \,\mathbb{O}_t^{\!\top} = [\omega(t)]_\times,
    \qquad \mathbb{O}_0 = \mathbb{I},
\end{equation}
with $[\omega]_\times$ the skew-symmetric matrix such that $[\omega]_\times v = \omega \times v$ for all $v \in \mathbb{R}^3$.

\begin{remark}
     The momentum equation of the rigid body reads
    \begin{equation}\label{eq:MomentumPDED}
        \partial_t(\chi_{\mathcal{B}_t}\rho_Bu_B)+\operatorname{div}\bigl(\chi_{\mathcal{B}_t}\rho_Bu_B\otimes u_B\bigr)=p_Fn\delta_{\partial\mathcal{B}_t}
    \end{equation}
   in the sense of distributions for test functions in $C_c^\infty((0,T);\mathcal{R}(\Omega))$ and 
\begin{equation}
    \mathcal{R}(\Omega):=\{\xi:\Omega\to \mathbb{R}^3 \big| \xi=\alpha+\eta\times(\cdot-X), \alpha,\eta,X\in \mathbb{R}^3 \}.
\end{equation}
Here $\chi_{\mathcal{B}_t}$ is the characteristic function over $\mathcal{B}_t$ and $\delta_{\partial\mathcal{B}_t}$ is the Dirac measure over $\partial\mathcal{B}_t$. We leave the proof for the Appendix.
\end{remark}

\subsection{Literature review and main results} 
Fluid-structure interaction arises in a wide range of natural systems and engineered structures, and it has long attracted the attention of mathematicians. An introduction to fluid-structure interaction models can be found in \cite{MR1942470}. As pointed out in \cite{gerard2014existence}, allowing for slip at the interface is natural in the context of viscous flows, since the classical no-slip condition leads to unrealistic collisional behavior between the solid and the domain boundary. The no-collision paradox was introduced in the context of the Stokes system in the 1960s \cite{MR161549,cooley1969slow}, and was confirmed at the level of the Navier-Stokes equations \cite{MR2354496,MR2481302}. Meanwhile, under Navier-type boundary conditions, collision can occur \cite{MR2665044,MR3281946}. 

Within the framework of Navier slip boundary conditions, several results have been established: the motion of a rigid body immersed in an incompressible Navier–Stokes flow in a bounded domain admits Leray-type weak solutions up to collision in three dimensions \cite{gerard2014existence}; this analysis was subsequently extended to polytropic compressible viscous flows, proving the local existence of weak solutions for adiabatic exponents greater than $\tfrac{3}{2}$ \cite{nevcasova2022motion}; and, using Lagrangian coordinates, local-in-time existence and uniqueness of strong solutions were obtained for polytropic compressible viscous flows with adiabatic exponents greater than 1 \cite{djebour2023existence}.  In the context of no-slip boundary conditions in viscous fluid-rigid systems, the existence of weak solutions up to collision was studied in \cite{MR1682663,MR1765138,MR1781915,MR1725677}. These results are also generalized to cases with possible collisions \cite{MR1870954,MR2019028,MR1981859}. Reducing the ideal fluid-rigid system to an ordinary differential equation on a closed infinite-dimensional manifold \cite{MR344713}, local in time existence and uniqueness of strong solutions was proved in \cite{MR2719277}. 

In terms of the vanishing viscosity limit, the only available result so far appears to be \cite{caggio2021measure}. There, in the incompressible situation, the vanishing viscosity limit for the model in \cite{gerard2014existence} is shown to be a dissipative Young measure-valued solution, and the corresponding weak-strong uniqueness was established. The need to consider measure-valued solutions instead of distributional ones stems from the notorious lack of compactness that impedes the passage to a weak solution of Euler, even in the absence of a moving rigid body, in the vanishing viscosity limit. See~\cite{wiedemannsurvey} for an introduction to measure-valued solutions and their weak-strong uniqueness.  

In this paper, we let the weak solutions for the compressible viscous fluid-rigid system \cite{nevcasova2022motion} converge, as viscosity tends to zero, to obtain a Young measure-valued solution to the compressible inviscid fluid-rigid system. We establish weak-strong uniqueness. A key novel ingredient (Step 2 in the proof of Theorem~\ref{thm:ExistenceofYMVS} below) is the delicate construction of a suitable approximation of the test function associated to viscosity. Indeed, the space of test functions in the weak formulation depends on the position of the rigid body, and thus on the viscosity parameter. We cannot therefore test the weak formulation for Euler and Navier--Stokes with a fixed test function. This issue has apparently been overlooked in~\cite{caggio2021measure}.     

We state our main results as follows. 

\begin{theorem}[Existence of Young Measure-Valued Solutions]\label{thm:ExistenceofYMVS}
Assume that the adiabatic exponent satisfies $\gamma > \tfrac{3}{2}$, and that the domains $\Omega$ and $\mathcal{B}_0$ are regular satisfying 
\[
\mathrm{dist}(\partial \Omega, \mathcal{B}_0) > 2\sigma
\]
for some positive constant $\sigma$. Suppose that the initial data fulfill the following conditions:
\begin{equation}\label{eq:initial-regularity}
\begin{aligned}
&\rho_F^0\in L^\gamma(\mathcal{F}_0), \qquad 
\rho_F^0 u_F^0\in L^{\frac{2\gamma}{\gamma+1}}(\mathcal{F}_0), \qquad
u_F^0 \chi_{\{\rho_F^0>0\}}\in L^2(\mathcal{F}_0),\\
&\rho_F^0 u_F^0 \chi_{\{\rho_F^0=0\}} = 0, \qquad
\rho_B^0\in L^\infty(\mathcal{B}_0), \qquad
u_B^0 = V_0 + \omega_0\times (x-X_0),\\
&\inf_{\mathcal{F}_0}\rho_F^0 \ge 0, \qquad 
\inf_{\mathcal{B}_0}\rho_B^0 > 0.
\end{aligned}
\end{equation}
Define the initial Young measure by
\[
\nu_0 = \chi_{\mathcal{F}_0}\,\delta_{(\rho_F^0, \sqrt{\rho_F^0}\,u_F^0)}
      + \chi_{\mathcal{B}_0}\,\delta_{(\rho_B^0, \sqrt{\rho_B^0}\,u_B^0)}.
\]

Then the compressible inviscid fluid-structure interaction system \eqref{eq:model}–\eqref{eq:inertia} admits a local-in-time dissipative Young measure-valued solution 
\[
(\nu,\,\mathcal{D}(t),\,\mathcal{B}_t)
\]
in the sense of Definition~\ref{def:definitionofYoungmeasure}.

Here and throughout the paper, we denote by $\chi_{\mathcal{Q}}$ the indicator function of a set $\mathcal{Q}$.
\end{theorem}

\begin{remark}
For a monatomic gas, the adiabatic exponent is $\gamma = \tfrac{5}{3}$, which satisfies the assumption $\gamma > \tfrac{3}{2}$ in Theorem~\ref{thm:ExistenceofYMVS}. In fact, we do not believe that $\gamma>\frac32$ is necessary for the existence of dissipative measure-valued solutions: The requirement stems from our approximation by solutions of the compressible Navier--Stokes equations, whose existence is only known for $\gamma>\frac32$. With a different (possibly physically less relevant) approximation, one may presumably recover a measure-valued solution in the full range $\gamma>1$. This has been done, for instance, in~\cite{MR3567640}. Note however the next theorem (weak-strong uniqueness) is proved for any $\gamma>1$.
\end{remark}

\begin{theorem}[Weak-Strong Uniqueness]\label{thm:Weak-strongUniqueness}
Suppose that the problem \eqref{eq:model}–\eqref{eq:inertia} admits a dissipative Young measure-valued solution 
\[
(\nu=\chi_{\mathcal{F}_{1,t}}\nu+\chi_{\mathcal{B}_{1,t}}\delta_{(\rho_{B1},\sqrt{\rho_{B1}}u_{B1})},~\mathcal{D}(t),~\mathcal{B}_{1,t})_{t\in[0,T]},
\]
and a strong solution 
\[
(\rho_{F2},u_{F2},\rho_{B2},u_{B2},\mathcal{B}_{2,t})_{t\in[0,T]},
\]
such that the following regularity and compatibility conditions hold:
\begin{equation}\label{eq:RestrictontestforstrongSolution}
\begin{aligned}
&u_{B1}(t,x)=V_1(t)+\omega_1(t)\times(x-X_1(t)), 
&&V_1(\cdot),\,\omega_1(\cdot)\in L^\infty(0,T);\\
&u_{B2}(t,x)=V_2(t)+\omega_2(t)\times(x-X_2(t)), 
&&V_2(\cdot),\,\omega_2(\cdot)\in C^1(0,T);\\
&\rho_{F2},u_{F2},\nabla_x\rho_{F2},\nabla_xu_{F2}\in C\!\left(\bigcup_{t\in[0,T]}\{t\}\times\overline{\mathcal{F}_{2,t}}\right), 
&&\partial_t\rho_{F2},\,\partial_tu_{F2}\in L^1\!\left([0,T];C(\overline{\mathcal{F}_{2,t}})\right);\\
&n\cdot(\nabla_xu_{F2}-\nabla_xu_{B2}) n\big|_{\partial\mathcal{B}_{2,t}}=0, 
&&\rho_{B2}\in C^1\!\left(\bigcup_{t\in[0,T]}\{t\}\times\overline{\mathcal{B}_{2,t}}\right);\\
&\inf_{t\in[0,T]}\inf_{x\in\mathcal{F}_{2,t}}\rho_{F2}(x,t)>0, 
&&\max_{i=1,2}\sup_{t\in[0,T]}\mathrm{dist}(\partial\Omega,\partial\mathcal{B}_{i,t})>\tfrac{3\sigma}{2}.
\end{aligned}
\end{equation}
Here, $n|_{\partial\mathcal{B}_{2,t}}$ denotes the unit outward normal vector at $\partial\mathcal{B}_{2,t}$.  

Then the dissipative measure-valued solution coincides with the strong solution in the sense that
\begin{equation}
\begin{aligned}
&\mathcal{B}_{1,t}=\mathcal{B}_{2,t}, &&t\in(0,T),\\
&\nu=\chi_{\mathcal{F}_{1,t}}\delta_{(\rho_{F2},\sqrt{\rho_{F2}}u_{F2})}
+\chi_{\mathcal{B}_{2,t}}\delta_{(\rho_{B2},\sqrt{\rho_{B2}}u_{B2})}, &&t\in(0,T),\\
&\mathcal{D}(t)=0, &&t\in(0,T).
\end{aligned}
\end{equation}

Here and in the sequel, we denote by $\overline{\mathcal{Q}}$ the closure of a set $\mathcal{Q}$.
\end{theorem}

Since the rigid body motion depends on the surrounding fluid velocity, and thus on viscosity, the admissible test functions $\phi$ in the weak formulation of the inviscid fluid-structure interaction system are generally not suitable for the viscous cases. Therefore, in the analysis of the vanishing viscosity limit, it is essential to construct appropriately modified test functions $\phi^\epsilon$ that remain compatible with the viscous problem. However, the nonlinear terms $\rho^\epsilon u^\epsilon\otimes u^\epsilon\chi_{\mathcal{F}_t^\epsilon}$ and $(\rho^\epsilon)^\gamma\chi_{\mathcal{F}_t^\epsilon}$ are uniformly bounded only in $L^1(\Omega)$, so even a slight modification of the test function within a thin layer near $\partial\mathcal{B}_t^\epsilon$ may affect the convergence behavior of the sequence
\begin{equation*}
    \int_0^T\int_{\Omega}\qnt{\rho^\epsilon u^\epsilon\otimes u^\epsilon\chi_{\mathcal{F}_t^\epsilon}+(\rho^\epsilon)^\gamma\chi_{\mathcal{F}_t^\epsilon}\mathbb{I}}:\nabla\phi^\epsilon \mathrm{d}x\mathrm{d}t. 
\end{equation*}
To ensure that the limiting measure-valued solution is independent of the particular construction of the test functions, and that the corresponding defect measures generated by these nonlinear terms are controlled by the energy dissipation $\mathcal{D}(t)$, we impose a stronger regularity requirement: in contrast to the viscous case, where only $C^0$ regularity of the normal component at the interface is needed, the inviscid formulation requires full $C^1$ regularity of the test functions across the fluid-solid interface.

The remaining part of the paper is organized as follows. 
In Section~2, we recall several fundamental results on generalized Young measures and formulate the precise definition of a measure-valued solution for the inviscid fluid-solid interaction system. We then verify that solutions to the viscous fluid-structure interaction problem with Navier slip boundary conditions converge, in the vanishing viscosity limit, to a dissipative Young measure-valued solution of the inviscid system. 
In Section~3, we introduce a suitable coordinate transformation to redefine the strong solution on the domain associated with the dissipative measure-valued solution. Using a detailed relative energy analysis, we show that the dissipative solution coincides with this redefined strong solution. Consequently, the redefined strong solution agrees with the original one, which establishes Theorem~\ref{thm:Weak-strongUniqueness}.

\section{Existence of Young Measure-Valued Solutions}
\subsection{Fundamental Results for Generalized Young Measures}

In this article, we impose no additional regularity assumptions beyond basic integrability on the initial data for the inviscid compressible fluid-rigid body system. Consequently, we formulate the problem within the framework of Young measure-valued solutions. We refer to \cite{gwiazda2015weak} and the references therein for the theory of Young measure-valued solutions to compressible Euler systems, particularly those capturing concentration effects. 

An arguably more economic formulation is the concept of a \emph{dissipative measure-valued solution} introduced in \cite{MR3567640}, where concentration effects are incorporated into a general defect measure.

Let $\mathcal{P}(\mathbb{R}^{l+m})$ denote the set of probability measures on $\mathbb{R}^{l+m}$, and let $\mathcal{M}(X)$ denote the set of finite Radon measures on a domain $X$. We recall the following fundamental result concerning the convergence of bounded sequences to Young measures.

\begin{lemma}\label{lemma:YoungMeasure1}
For $1<p,q<\infty$, denote by $C_{p,q}$ the space of functions $f\in C(\mathbb{R}^l \times \mathbb{R}^m; \mathbb{R})$ such that there exists $C>0$ for which
\begin{equation*}
|f(\lambda_1,\lambda')|\leq C(1+|\lambda_1|^\gamma+|\lambda'|^2)    
\end{equation*}
for all $(\lambda_1,\lambda')\in \mathbb{R}^l \times \mathbb{R}^m$.

Let $\{(u_k, w_k)\}_{k\in\mathbb{N}} \subset \mathbb{R}^{l+m}$ be a sequence uniformly bounded in $L^{\infty}\big([0,T]; L^p(\Omega)\times L^q(\Omega)\big)$. Then there exist a subsequence (not relabeled), a Young measure $\nu \in L^{\infty}_w(\Omega; \mathcal{P}(\mathbb{R}^{l+m}))$, and a family of concentration defect measures
\[
m_t^{(\cdot)}: C_{p,q} \longrightarrow \mathcal{M}(\overline{\Omega}),
\]
such that 
\begin{equation}\label{eq:YoungmeasureConvergence}
    f(u_n(t,x), w_n(t,x))\,\mathcal{L} \overset{*}{\rightharpoonup}
    \int_{\mathbb{R}^{l+m}} f(\lambda_1, \lambda')\,\mathrm{d}\nu\,\mathcal{L}
    + m^f_t \otimes \mathrm{d}t,
\end{equation}
for any $f \in C_{p,q}$. 
Here $\mathcal{L}$ denotes the Lebesgue measure on $\mathbb{R}^4$, and $\mathrm{d}t$ is the Lebesgue measure on $\mathbb{R}$.

Furthermore, if 
\[
\sup_{n\in\mathbb{N}} \| f(u_n, w_n) \|_{L^{r}((0,T)\times\Omega)} < +\infty
\]
for some $r > 1$, then
\begin{equation}
    m^f_t = 0.
\end{equation}
\end{lemma}

Moreover, following \cite[Lemma 2.1]{MR3567640}, the following comparison principle for concentration defect measures holds.

\begin{lemma}\label{lemma:comparison}
Let $\{(u_k,w_k)\}_{k\in \mathbb{N}}, \nu,$ and $m^f$ be as in Lemma~\ref{lemma:YoungMeasure1}. 
Let $G \in C_0(\mathbb{R}^{l+m}; \mathbb{R}^+)$ and $F \in C_0(\mathbb{R}^{l+m}; \mathbb{R})$ satisfy
\[
F(z) \le |G(z)|, \quad \forall\, z \in \mathbb{R}^{l+m},
\]
and assume
\[
\sup_{n\in \mathbb{N}} \| G(u_n, w_n) \|_{L^1([0,T]\times\Omega)} < +\infty.
\]
Then the corresponding concentration defect measures satisfy
\begin{equation}
    |m^F| \le m^G.
\end{equation}
\end{lemma}

\subsection{Definition of Dissipative Young Measure-Valued Solutions}

For notational simplicity, we define
\[
\mathcal{H}^T := \bigcup_{s\in(0,T)} \{s\}\times\mathcal{H}_s,
\]
where $\mathcal{H}_t$ is a time-dependent domain. In particular, we will use the following shorthand:
\[
\Omega^T, \ \overline{\Omega^T}, \ \mathcal{F}^T, \ \overline{\mathcal{F}^T}, \ \mathcal{B}^T.
\]

We further define, for $(\lambda_1, \lambda') \in \mathbb{R}^{1+3}$ and $\nu \in L^{\infty}_w(\Omega; \mathcal{P}(\mathbb{R}^{1+3}))$,
\[
\langle f(\lambda_1, \lambda'), \nu \rangle := \int_{\mathbb{R}^{1+3}} f(\lambda_1, \lambda') \, \mathrm{d}\nu.
\]

Next, we introduce the class of admissible test functions used in the definition of weak solution to the viscous fluid-structure interaction model:
\begin{equation}\label{eq:DefofTestFunction}
    \begin{aligned}
 \mathcal{V}_{\mathcal{B}^T} :=\left\{
\phi \in C([0,T]; L^2(\Omega)) \; \middle|\; 
\begin{aligned}
&\text{there exist } \phi^F \in C^\infty([0,T]; C^\infty(\overline{\Omega})),\\& 
\phi^B \in C^\infty([0,T]; \mathcal{R}(\Omega)) \text{ such that:~} \\& \phi\big|_{\mathcal{F}_t} = \phi^F, 
\phi\big|_{ \mathcal{B}_t} = \phi^B, \phi^F \cdot n \big|_{\partial \Omega}= 0 \\
& (\phi^F-\phi^B) \cdot n\big|_{\partial \mathcal{B}_t} = 0,
\end{aligned}
\right\},
    \end{aligned}
\end{equation}
where we recall
\[
\mathcal{R}(\Omega) := \{\xi:\Omega\to \mathbb{R}^3 \big| \xi=\alpha+\eta\times(\cdot-X), \alpha,\eta,X\in \mathbb{R}^3 \},
\]
and $n$ denotes the unit outward normal vector to $\partial \mathcal{B}_t$. Furthermore, we define the following test functions used in the definition of dissipative Young measure-valued solution to our inviscid fluid-structure interaction model:
\begin{equation}
    \tilde{\mathcal{V}}_{\mathcal{B}^T}:=\{\phi\in \mathcal{V}_{\mathcal{B}^T}\big| n\cdot\nabla(\phi^F-\phi^B)n\big|_{\partial \mathcal{B}_t}=0\}.
\end{equation} 

\begin{remark}
For any $\phi \in \tilde{\mathcal{V}}_{\mathcal{B}^T}$, the constraints 
\[
(\phi^F - \phi^B)\cdot n|_{\partial\mathcal{B}_t} = 0, \quad
n \cdot \nabla(\phi^F - \phi^B)n|_{\partial\mathcal{B}_t} = 0,
\]
imply that $\phi \cdot n$ possesses $C^1$ regularity across $\partial\mathcal{B}_t$.
\end{remark}

We are now ready to state the precise definition of a dissipative Young measure-valued solution for the inviscid compressible fluid-rigid body system.

\begin{definition}\label{def:definitionofYoungmeasure}
A triple $(\nu, \mathcal{D}, \mathcal{B}_t)$ is called a \emph{dissipative Young measure-valued solution} to the system \eqref{eq:model}–\eqref{eq:ic} on $t \in [0,T]$ if the following properties hold:

\begin{enumerate}[(1)]
\item \textbf{Basic regularity:}  
$\nu \in L^{\infty}_w(\Omega^T; \mathcal{P}(\mathbb{R}^{1+3}))$, 
$\mathcal{D} \in L^\infty([0,T])$ with $\mathcal{D} \ge 0$, 
and $\mathcal{B}_t$ is a time-dependent domain.

\item \textbf{Rigid motion of $\mathcal{B}_t$:}  
There exist absolutely continuous functions $V(t)$ and $\omega(t)$ such that
\[
\nu|_{\mathcal{B}^T} = \delta_{(\rho_B(t,x), \sqrt{\rho_B(t,x)}\,u_B(t,x))},
\]
where $u_B(t,x) = V(t) + \omega(t)\times (x - X(t))$. For every $\phi(t,x) \in C^{\infty}(\overline{\Omega^T}; \mathbb{R})$,
\begin{equation}\label{eq:MassCRigidP}
\begin{aligned}
&\iint_{\mathcal{B}^t} \rho_B \, \phi_\tau + \rho_B u_B \cdot \nabla \phi \, \mathrm{d}x\,\mathrm{d}\tau
+ \int_{\mathcal{B}_0} \rho_0 \phi(0,x)\,\mathrm{d}x
- \int_{\mathcal{B}_t} \rho_B \phi(t,x)\,\mathrm{d}x = 0, \\
&\iint_{\mathcal{B}^t} \phi_\tau + u_B \cdot \nabla \phi \, \mathrm{d}x\,\mathrm{d}\tau
+ \int_{\mathcal{B}_0} \phi(0,x)\,\mathrm{d}x
- \int_{\mathcal{B}_t} \phi(t,x)\,\mathrm{d}x = 0.
\end{aligned}
\end{equation}

\item \textbf{Conservation of mass:}  
%There exists $\nu^C \in L^1([0,T]; \mathcal{M}(\overline{\mathcal{F}_t}))$ satisfying
%\[
%\Big|\int_{\mathcal{F}_t} \nabla %\psi \, \mathrm{d}\nu^C\Big|
%\le \zeta_C(t)\,\mathcal{D}(t)\,\|\psi(t,\cdot)\|_{C^1(\mathcal{F}_t)},
%\]
%for some $\zeta_C(\cdot) \in L^1([0,T])$, such that 
For every $\psi(t,x) \in C^{\infty}(\overline{\Omega^T}; \mathbb{R})$,
\begin{equation}\label{MVS MassC}
\begin{aligned}
0 &= \iint_{\mathcal{F}^t} \langle \lambda_1, \nu \rangle \psi_\tau \mathrm{d}x\,\mathrm{d}\tau 
  + \langle \sqrt{\lambda_1}\lambda', \nu \rangle \cdot \nabla \psi \, \mathrm{d}x\,\mathrm{d}\tau \\
&\quad + \int_{\mathcal{F}_0} \langle \lambda_1, \nu_0 \rangle \psi(0,x)\,\mathrm{d}x
 - \int_{\mathcal{F}_t} \langle \lambda_1, \nu \rangle \psi(t,x)\,\mathrm{d}x.
\end{aligned}
\end{equation}

\item \textbf{Balance of momentum:}  
There exists $\nu^M \in L^1([0,T]; \mathcal{M}(\overline{\mathcal{F}_t}))$ satisfying
\begin{equation}\label{ersatzpoincare}
\Big|\int_{\mathcal{F}_t} \nabla \Phi^F \, \mathrm{d}\nu^M\Big|
\le \zeta(t)\,\mathcal{D}(t)\,\|\Phi^F(t,\cdot)\|_{C^1(\mathcal{F}_t)},
\end{equation}
for some $\zeta(\cdot) \in L^1([0,T])$, such that for any $\Psi \in \tilde{\mathcal{V}}_{\mathcal{B}^T}$ with corresponding $(\Psi^F, \Psi^B)$ from \eqref{eq:DefofTestFunction},
\begin{equation}\label{MVS MomentumC}
\begin{aligned}
0 &= \iint_{\mathcal{F}^t}
    \langle \sqrt{\lambda_1}\lambda', \nu \rangle \Psi^F_\tau
   + \langle \lambda'\otimes\lambda' + \lambda_1^\gamma \mathbb{I}, \nu \rangle
      \!:\! \nabla \Psi^F \, \mathrm{d}\nu\,\mathrm{d}\tau
   + \iint_{\mathcal{B}^t} \rho_B u_B \Psi^B_\tau \,\mathrm{d}x\,\mathrm{d}t
   + \iint_{\mathcal{F}^t} \nabla \Psi^F \, \mathrm{d}\nu^M\,\mathrm{d}\tau \\
&\quad + \int_{\Omega} \langle \sqrt{\lambda_1}\lambda', \nu_0 \rangle \Psi(0,x)\,\mathrm{d}x
 - \int_{\mathcal{F}_t} \langle \sqrt{\lambda_1}\lambda', \nu \rangle \Psi^F(t,x)\,\mathrm{d}x
 - \int_{\mathcal{B}_t} \langle \sqrt{\lambda_1}\lambda', \nu \rangle \Psi^B(t,x)\,\mathrm{d}x,
\end{aligned}
\end{equation}
where
\[
\nu_0 = \delta_{(\rho_F^0, \sqrt{\rho_F^0}\,u_F^0)} \chi_{\mathcal{F}_0}
+ \delta_{(\rho_B^0, \sqrt{\rho_B^0}(V_0 + \omega_0 \times (\cdot-X_0)))} \chi_{\mathcal{B}_0}.
\]

\item \textbf{Dissipation of kinetic energy:}  
For all $t \in [0,T]$,
\begin{equation}\label{eq:EnergyMVS}
\begin{aligned}
E_{mvs}(t)
&:= \int_{\mathcal{F}_t} \Big\langle \frac{|\lambda'|^2}{2} + \frac{\lambda_1^\gamma}{\gamma - 1}, \nu \Big\rangle \mathrm{d}x
   + \int_{\mathcal{B}_t} \frac{\rho_B |u_B|^2}{2} \, \mathrm{d}x
   + \mathcal{D}(t) \\
&\le \int_{\mathcal{F}_0} \frac{(\rho^0_F)^\gamma}{\gamma-1} \mathrm{d}x
   + \int_{\Omega} \frac{\rho_F^0 |u^0_F|^2}{2} \mathrm{d}x.
\end{aligned}
\end{equation}
\end{enumerate}
\end{definition}

\begin{remark}
By combining the momentum equation for the fluid in \eqref{eq:model} with the momentum equation for the rigid body stated in Theorem~\ref{thm:ApendixA} in Appendix \ref{App:WeakformofSolid}, we derive the weak formulation of the momentum balance \eqref{MVS MomentumC} for the coupled fluid-rigid body system.
\end{remark}

\begin{remark}
The test functions in the definition of dissipative Young measure-valued solutions can be extended, by a straightforward argument, to the following function classes:
\begin{equation*}
\begin{aligned}
&\phi, \nabla_x \phi \in C\Bigl(\overline{\mathcal{B}^T}\Bigl), \quad
\partial_t \phi \in L^1\bigr([0,T]; C(\overline{\mathcal{B}_t})\bigl) 
\quad \text{in \eqref{eq:MassCRigidP}},\\
&\psi, \nabla_x \psi \in C\Bigl(\overline{\mathcal{F}^T}\Bigr), \quad
\partial_t \psi \in L^1\bigr([0,T]; C(\overline{\mathcal{F}_t})\bigl) 
\quad \text{in \eqref{MVS MassC}},
\end{aligned}
\end{equation*}
and
\begin{equation}
\begin{aligned}
&\Psi^F, \nabla_x \Psi^F \in C\Bigl(\overline{\mathcal{F}^T}\Bigr), \quad
\partial_t \Psi^F \in L^1\bigr([0,T]; C(\overline{\mathcal{F}_t})\bigl),\\
&\Psi^B, \nabla_x \Psi^B \in C\Bigl(\overline{\mathcal{B}^T}\Bigr), \quad
\partial_t \Psi^B \in L^1\bigr([0,T]; C(\overline{\mathcal{B}_t})\bigl),\\
&(\Psi^F - \Psi^B)\cdot n|_{\partial\mathcal{B}_t} = 0, \quad
n \cdot \nabla(\Psi^F - \Psi^B)n|_{\partial\mathcal{B}_t} = 0
\quad \text{in \eqref{MVS MomentumC}}.
\end{aligned}
\end{equation}
Therefore, it is admissible to use the strong solution as a test function in the proof of Theorem~\ref{thm:Weak-strongUniqueness}.
\end{remark}

\subsection{Existence of Young Measure-Valued Solutions}

By a vanishing-viscosity limit of the compressible viscous fluid-rigid body system with Navier-slip boundary conditions, we obtain a dissipative Young measure-valued solution for our fluid-structure model.

The compressible viscous fluid-structure system with Navier boundary conditions is
\begin{equation}\label{eq:Viscoussystem}
\left\{
\begin{aligned}
&\partial_t \rho_F^\epsilon + \operatorname{div}_x(\rho_F^\epsilon u_F^\epsilon) = 0,\\
&\partial_t (\rho_F^\epsilon u_F^\epsilon) + \operatorname{div}_x(\rho_F^\epsilon u_F^\epsilon \otimes u_F^\epsilon + p_F^\epsilon \mathbb{I})
  - \epsilon \operatorname{div}_x \mathcal{T}(u_F^\epsilon) = 0,
  && t \in (0,T),\ x \in \mathcal{F}_t^\epsilon,\\
&\partial_t \rho_B^\epsilon + u_B^\epsilon \cdot \nabla \rho_B^\epsilon = 0,\\
&\frac{\mathrm{d}X^\epsilon}{\mathrm{d}t}(t) = V^\epsilon(t),\\
&M \frac{\mathrm{d}V^\epsilon}{\mathrm{d}t}(t)
  = \int_{\partial \mathcal{B}_t^\epsilon} \big( p_F^\epsilon \mathbb{I} - \epsilon \mathcal{T}(u_F^\epsilon) \big)\cdot n \, \mathrm{d}S,\\
&\big(\mathcal{J}^\epsilon(t)\,\omega^\epsilon(t)\big)'
  = \int_{\partial \mathcal{B}_t^\epsilon} (x - X^\epsilon(t)) \times
     \big( (p_F^\epsilon \mathbb{I} - \epsilon \mathcal{T}(u_F^\epsilon))\, n \big)\, \mathrm{d}S,
  && t \in (0,T),\\
&u_F^\epsilon \cdot n = 0,\quad
  (\mathcal{T}(u_F^\epsilon)\cdot n)\times n = -\epsilon (u_F^\epsilon \times n),
  && t \in (0,T),\ x \in \partial \Omega,\\
&(u_F^\epsilon - u_B^\epsilon)\cdot n = 0,\quad
  (\mathcal{T}(u_F^\epsilon)\cdot n)\times n = -\epsilon \big((u_F^\epsilon - u_B^\epsilon)\times n\big),
  && t \in (0,T),\ x \in \partial \mathcal{B}_t^\epsilon,
\end{aligned}
\right.
\end{equation}
where
\[
p_F^\epsilon = (\rho_F^\epsilon)^\gamma,\qquad
\mathcal{T}(u_F^\epsilon) = \nabla u_F^\epsilon + (\nabla u_F^\epsilon)^{\!\top}
+ (\operatorname{div} u_F^\epsilon)\, \mathbb{I},
\]
and the relations among $\qnt{V^\epsilon, \omega^\epsilon, X^\epsilon, u_B^\epsilon, \mathcal{J}^\epsilon}$ are as in the inviscid case for $\qnt{V,\omega,X,u_B,\mathcal{J}}$.

The weak solution theory for \eqref{eq:Viscoussystem} is established in \cite{nevcasova2022motion}; we summarize it below.

\begin{lemma}\label{lem:ExsitenceofViscousSystem}
Let $\Omega \subset \mathbb{R}^3$ and $\mathcal{B}_0 \subset \Omega$ be regular bounded domains. 
Assume that for some $\sigma > 0$,
\[
\operatorname{dist}(\mathcal{B}_0, \partial \Omega) > 2\sigma.
\]
Let $\gamma > \tfrac{3}{2}$, and suppose that the initial data satisfy
\begin{equation*}
\begin{aligned}
&\rho_F^\epsilon\big|_{t=0} \in L^{\gamma}(\mathcal{F}_0), 
\qquad \rho_F^\epsilon(0,x) \ge 0, 
\qquad \inf\rho_B^\epsilon\big|_{t=0} > 0,\\[2pt]
&\rho_F^\epsilon u_F^\epsilon\big|_{t=0} \in L^{\frac{2\gamma}{\gamma+1}}(\mathcal{F}_0), 
\qquad \big(\rho_F^\epsilon u_F^\epsilon\, 1_{\{\rho_F^\epsilon=0\}}\big)\big|_{t=0} = 0,\\[2pt]
&|u_F^\epsilon|^2\big|_{t=0} \in L^1(\mathcal{F}_0), 
\qquad u_B^\epsilon\big|_{t=0} = \ell_0 + \omega_0 \times (x-X_0), \quad x \in \mathcal{B}_0,
\quad \text{for some }\ell_0,\omega_0 \in \mathbb{R}^3.
\end{aligned}
\end{equation*}

Then there exists a time $T>0$, independent of $\epsilon$, such that system~\eqref{eq:Viscoussystem} admits a finite-energy weak solution on $[0,T)$ satisfying
\[
\mathcal{B}_t^\epsilon \subset \Omega, 
\qquad \operatorname{dist}(\mathcal{B}_t^\epsilon, \partial \Omega) > \tfrac{3\sigma}{2}, 
\quad t \in [0,T).
\]

A triple $(\mathcal{B}_t^\epsilon, \rho^\epsilon(t,x), u^\epsilon(t,x))$ is called a \emph{finite-energy weak solution} to~\eqref{eq:Viscoussystem} if it satisfies the following conditions:

\begin{itemize}
\item[1.] For $t \in [0,T)$,
\[
\mathcal{B}_t^\epsilon \subset \Omega,\quad
\operatorname{dist}(\mathcal{B}_t^\epsilon, \partial \Omega) > 0,\quad
\chi_{\mathcal{B}^\epsilon}(t,x)  \in L^\infty\big([0,T)\times\Omega\big).
\]

\item[2.] $u^\epsilon \in \mathcal{U}_{\mathcal{B}^{\epsilon T}}$, where
\begin{equation}\label{eq:DefofViscosVelocityField}
    \begin{aligned}
 \mathcal{U}_{\mathcal{B}^{\epsilon T}} :=\left\{
u \in L^2([0,T]; L^2(\Omega)) \; \middle|\; 
\begin{aligned}
&\text{there exist } u^F \in L^2([0,T]; H^1(\overline{\Omega})),\\& 
u^B \in L^2([0,T]; \mathcal{R}(\Omega)) \text{ such that:~} \\& u\big|_{\mathcal{F}^{\epsilon T}} = u^F, 
u\big|_{ \mathcal{B}^{\epsilon T}} = u^B, u^F \cdot n \big|_{\partial \Omega}= 0 \\
& u^F \cdot n\big|_{\partial \mathcal{B}^{\epsilon T}} = u^B \cdot n\big|_{\partial \mathcal{B}^{\epsilon T}}
\end{aligned}
\right\}.
    \end{aligned}
\end{equation}

\item[3.] $\rho^\epsilon \ge 0$, $\ \rho^\epsilon \in L^\infty(0,T; L^\gamma(\Omega))$, and $\ \rho^\epsilon |u^\epsilon|^2 \in L^\infty(0,T; L^1(\Omega))$.

\item[4.] For any $\phi_i \in C_c^\infty(\Omega^T;\mathbb{R})$, $i=1,2,3$,
\begin{equation}\label{eq:DefofContinuityEquation}
\begin{aligned}
0 =\ &\int_0^t \!\!\int_{\mathcal{B}_\tau^\epsilon} \rho^\epsilon \,\partial_\tau \phi_1
+ \rho^\epsilon u^\epsilon \cdot \nabla \phi_1 \, \mathrm{d}x\,\mathrm{d}\tau
+ \int_{\mathcal{B}_0} \rho^\epsilon(0,x) \phi_1(0,x)\,\mathrm{d}x\\
&+ \int_0^t \!\!\int_{\mathcal{B}_\tau^\epsilon} \partial_\tau \phi_2
+ u^\epsilon \cdot \nabla \phi_2 \, \mathrm{d}x\,\mathrm{d}\tau
+ \int_{\mathcal{B}_0} \phi_2(0,x)\,\mathrm{d}x\\
&+ \int_0^t \!\!\int_{\mathcal{F}_\tau^\epsilon} \rho^\epsilon \,\partial_\tau \phi_3
+ \rho^\epsilon u^\epsilon \cdot \nabla \phi_3 \, \mathrm{d}x\,\mathrm{d}\tau
+ \int_{\mathcal{F}_0} \rho^\epsilon(0,x) \phi_3(0,x)\,\mathrm{d}x\\
&- \int_{\mathcal{B}_t^\epsilon} \rho^\epsilon(t,x) \phi_1(t,x)\,\mathrm{d}x
- \int_{\mathcal{B}_t^\epsilon} \phi_2(t,x)\,\mathrm{d}x
- \int_{\mathcal{F}_t^\epsilon} \rho^\epsilon(t,x) \phi_3(t,x)\,\mathrm{d}x.
\end{aligned}
\end{equation}

\item[5.] For any $\phi \in \mathcal{V}_{\mathcal{B}^{\epsilon T}}$ with corresponding $(\phi^F,\phi^B)$ as in \eqref{eq:DefofTestFunction},
\begin{equation}\label{eq:ViscousMomentumEstimates}
\begin{aligned}
&- \int_0^t \!\!\int_{\mathcal{F}_\tau^\epsilon} \rho^\epsilon u^\epsilon \cdot \partial_\tau \phi^F
  - \int_0^t \!\!\int_{\mathcal{B}_\tau^\epsilon} \rho^\epsilon u^\epsilon \cdot \partial_\tau \phi^B
  - \int_0^t \!\!\int_{\mathcal{F}_\tau^\epsilon} (\rho^\epsilon u^\epsilon \otimes u^\epsilon) : \nabla \phi^F \\
&\quad + \int_0^t \!\!\int_{\mathcal{F}_\tau^\epsilon} \big( \epsilon \mathcal{T}(u^\epsilon) - p_F^\epsilon \mathbb{I} \big) : \nabla \phi^F
  + \epsilon \int_0^t \!\!\int_{\partial \Omega} (u_F^\epsilon \times n)\cdot (\phi^F \times n) \, \mathrm{d}S\,\mathrm{d}\tau \\
&\quad + \epsilon \int_0^t \!\!\int_{\partial \mathcal{B}_\tau^\epsilon}
          \big( (u_F^\epsilon - u_B^\epsilon)\times n \big)\cdot \big( (\phi^F - \phi^B)\times n \big)\, \mathrm{d}S\,\mathrm{d}\tau \\
&= \int_{\Omega} \rho^\epsilon u^\epsilon \cdot \phi \big|_{\tau=0}\, \mathrm{d}x
 - \int_{\Omega} \rho^\epsilon u^\epsilon \cdot \phi \big|_{\tau=t}\, \mathrm{d}x.
\end{aligned}
\end{equation}

\item[6.] (Energy inequality)
\begin{equation}\label{eq:EnergyInequalityofVisSys}
\begin{aligned}
E^\epsilon(t)
&+ \epsilon \int_0^t \!\!\int_{\mathcal{F}_\tau^\epsilon} \big( |\nabla u_F^\epsilon|^2 + |\operatorname{div} u_F^\epsilon|^2 \big)\,\mathrm{d}x\,\mathrm{d}\tau
+ \epsilon \int_0^t \!\!\int_{\partial \Omega} |u_F^\epsilon \times n|^2 \,\mathrm{d}S\,\mathrm{d}\tau\\
&+ \epsilon \int_0^t \!\!\int_{\partial \mathcal{B}_\tau^\epsilon} |(u_F^\epsilon - u_B^\epsilon)\times n|^2 \,\mathrm{d}S\,\mathrm{d}\tau
\ \le\ E^\epsilon(0),
\end{aligned}
\end{equation}
where
\[
E^\epsilon(t) := \int_{\Omega} \frac{\rho^\epsilon |u^\epsilon|^2}{2}\, \mathrm{d}x
+ \int_{\mathcal{F}_t^\epsilon} \frac{(\rho^\epsilon)^\gamma}{\gamma-1}\, \mathrm{d}x.
\]
\end{itemize}
\end{lemma}

We improve the bounds on $u_B^\epsilon$ and $\rho_B^\epsilon$ from Lemma~\ref{lem:ExsitenceofViscousSystem}.

\begin{lemma}
It holds that
\begin{equation}\label{eq:BoundsonDesity}
0<\inf_{\mathcal{B}_0} \rho_B^0(x) \ \le \ \rho_B^\epsilon \ \le \ \sup_{\mathcal{B}_0} \rho_B^0(x),
\end{equation}
and $u_B^\epsilon \in L^\infty\big([0,T); \mathcal{R}(\Omega)\big)$.
\end{lemma}

\begin{proof}
Consider
\begin{equation}\label{eq:AP1}
\begin{aligned}
&\partial_t \varrho + {\tilde{u}}\cdot \nabla \varrho = 0,\quad (t,x)\in [0,T)\times \Omega,\\
&\varrho\big|_{t=0} = \varrho_0(x) :=
\begin{cases}
\rho_B^0(x) - \inf_{\mathcal{B}_0} \rho_B^0(x), & x \in \mathcal{B}_0,\\
0, & x \in \mathbb{R}^3 \setminus \mathcal{B}_0,
\end{cases}
\end{aligned}
\end{equation}
with $\tilde{u} \in L^2\big([0,T); \mathcal{R}(\Omega)\big)$. Note that $\varrho_0 \in L^\infty(\mathbb{R}^3)$.

By the weak formulation of the $\rho_B^\epsilon$–equation in \eqref{eq:DefofContinuityEquation}, we observe that
\[
\varrho(t,x) =
\begin{cases}
\rho_B^\epsilon(t,x) - \inf_{\mathcal{B}_0} \rho_B^0(x), & (t,x) \in \mathcal{B}^{\epsilon T},\\
0, & (t,x) \in [0,T)\times (\mathbb{R}^3 \setminus \mathcal{B}^{\epsilon T}),
\end{cases}
\]
is a weak solution to \eqref{eq:AP1}.

By the DiPerna–Lions theory \cite[Corollary II.1]{diperna1989ordinary}, \eqref{eq:AP1} admits a unique weak solution in $L^\infty\big([0,T); L^\gamma(\mathbb{R}^n)\big)$ and satisfies
\[
0 \ \le \ \varrho(t,x) \ \le \ \sup_{\mathcal{B}_0} \rho_B^0(x) - \inf_{\mathcal{B}_0} \rho_B^0(x)
\quad \text{for a.e. } (t,x)\in [0,T)\times \mathbb{R}^3.
\]
This implies \eqref{eq:BoundsonDesity}. We now bound $u_B^\epsilon$.

Using $u_B^\epsilon = V^\epsilon(t) + \omega^\epsilon(t)\times (x - X^\epsilon(t))$ and the lower bound on $\rho_B^\epsilon$, the energy inequality \eqref{eq:EnergyInequalityofVisSys} yields, for some $C>0$,
\[
\begin{aligned}
E^\epsilon(0)
&\ge \int_{\mathcal{B}_t^\epsilon} \rho_B^\epsilon |u_B^\epsilon|^2 \,\mathrm{d}x
\ \ge\ \inf_{\mathcal{B}_0} \rho_B^0 \int_{\mathcal{B}_t^\epsilon} |u_B^\epsilon|^2 \,\mathrm{d}x \\
&\ge C\Big( |\alpha^\epsilon(t)|^2 + |\eta^\epsilon(t)|^2 \Big),
\end{aligned}
\]
hence $u_B^\epsilon \in L^\infty\big([0,T); \mathcal{R}(\Omega)\big)$. This concludes the proof.
\end{proof}

Having recalled the viscous theory, we are ready to pass to the limit $\epsilon \to 0$ in \eqref{eq:Viscoussystem} to prove Theorem~\ref{thm:ExistenceofYMVS} (existence of dissipative Young measure-valued solutions).

\begin{proof}[Proof of Theorem \ref{thm:ExistenceofYMVS}]
We divide the proof into 4 parts. 

\smallskip
\textbf{Step 1: The convergence on solid parts.} 
Recalling $u^\epsilon_B\in L^\infty([0,T);\mathcal{R}(\Omega))$, by \cite[Proposition 3.4]{gerard2014existence} and \cite[Theorem II.4]{diperna1989ordinary}, there exist $u_B\in L^\infty([0,T);\mathcal{R}(\Omega))$, $\rho_B\in L^\infty([0,T)\times\Omega)$, a bounded time-dependent domain $\mathcal{B}_t$, and a subsequence of $u^\epsilon_B,\rho^\epsilon_B$ and $\chi_{\mathcal{B}^\epsilon}$ (not relabeled) such that 
\begin{equation}\label{eq:SolidConvergence1}
    u_B^\epsilon \xrightarrow{\text{weakly-$*$ in }L^\infty([0,T);L^2(\Omega))} u_B,
\end{equation}
\begin{equation}\label{eq:convergenceofB}
    \chi_{\mathcal{B}^{\epsilon T}}(t,x) 
    \xrightarrow[\text{weakly-$*$ in }L^\infty([0,T)\times \mathbb{R}^3 )]{\text{strongly in } C([0,T);L^p_{\text{loc}}(\mathbb{R}^3 ))} 
    \chi_{\mathcal{B}^T}(t,x)=:1_{\mathcal{B}_t}(x),
\end{equation}
and 
\begin{equation}\label{eq:SolidConvergence111}
     \rho_B^\epsilon(t,x) 
     \xrightarrow[\text{weakly-$*$ in }L^\infty([0,T)\times \mathbb{R}^3 )]{\text{strongly in } C([0,T);L^p_{\text{loc}}(\mathbb{R}^3 ))} 
     \rho_B(t,x).
\end{equation}
Moreover, for any $\phi_1,\phi_2\in C_c^\infty(\Omega^T;\mathbb{R})$,
\begin{equation}\label{eq:ConvergenceofSolidDensity} 
\begin{aligned}
0&=\int_0^t\!\!\int_{\mathcal{B}_\tau}\partial_\tau\phi_2+u_B\cdot\nabla\phi_2\,\mathrm{d}x\,\mathrm{d}\tau
+\int_{\mathcal{B}_0}\phi_2(0,x)\,\mathrm{d}x
-\int_{\mathcal{B}_t}\phi_2(t,x)\,\mathrm{d}x\\
&\quad+\int_0^t\!\!\int_{\mathcal{B}_\tau}\rho_B\,\partial_\tau\phi_1+\rho_B u_B\cdot\nabla\phi_1\,\mathrm{d}x\,\mathrm{d}\tau
+\int_{\mathcal{B}_0}\rho_B(0,x)\phi_1(0,x)\,\mathrm{d}x
-\int_{\mathcal{B}_t}\rho_B(t,x)\phi_1(t,x)\,\mathrm{d}x.
\end{aligned}
\end{equation}

We recall \cite[Lemma 5.4]{gerard2014existence} together with the rigid motion defined in \eqref{eq:body-domain}. As a consequence, we establish the convergence properties of the moving boundaries $\partial\mathcal{B}_t^\epsilon \to \partial\mathcal{B}_t$ in the following lemma.

\begin{lemma}\label{lem:ConvergeBou1}
It holds that
\begin{equation}\label{eq:ConvergeBou}
  \lim_{\epsilon \to 0} 
  \sup_{t \in [0,T]} 
  \mathcal{L}^3\!\Big(\mathcal{B}_t^\epsilon \Delta \mathcal{B}_t\Big) = 0,
\end{equation}
where $\Delta$ denotes symmetric difference. Let
\begin{equation}
    G:=\{Q\in SO(3):\, Q\mathcal{B}_0=\mathcal{B}_0\}
\end{equation}
be the rotational symmetry group of $\mathcal{B}_0$. Then for any $\delta \in (0,1)$ there exists $\epsilon_\delta > 0$ such that, for all $\epsilon \in (0,\epsilon_\delta]$, the following holds:
\begin{equation}\label{eq:ConvergeBoundary}
\begin{aligned}
&\sup_{t\in[0,T]}\ \inf_{Q\in G}\ \sup_{x\in\partial\mathcal{B}_0}
\big|y^\varepsilon(t,x)-y(t,Qx)\big|
\le \mathcal{O}(1)\,\delta^5,\\[2pt]
&\sup_{t\in[0,T]}\ \inf_{Q\in G}\ \sup_{x\in\partial\mathcal{B}_0}
\big|n(y^\varepsilon(t,x))-n(y(t,Qx))\big|
\le \mathcal{O}(1)\,\delta^5,
\end{aligned}
\end{equation}
where 
\[
y^\epsilon(t,x) = X^\epsilon(t) + \mathbb{O}_t^\epsilon x, 
\qquad 
y(t,x) = X(t) + \mathbb{O}_t x, 
\quad x \in \partial\mathcal{B}_0,
\]
and $n(y^\epsilon(t,x))$ and $n(y(t,x))$ denote the outward unit normal vectors to $\partial\mathcal{B}_t^\epsilon$ and $\partial\mathcal{B}_t$, respectively, at the corresponding points.
\end{lemma}

The convergence in \eqref{eq:ConvergeBou} follows directly from \cite[Lemma~5.4]{gerard2014existence}.
The convergence of the moving boundaries
and of the associated outward unit normals, up to rigid symmetries, relies on
a compactness argument on the symmetry group $G$ of the rigid body. The detailed proof is given in Appendix~\ref{app:proof-lem-ConvergeBou1}.

\smallskip
\textbf{Step 2: The approximation of test functions.} 
Because the solid position depends on $\epsilon$, the set $\mathcal{V}_{\mathcal{B}^{\epsilon T}}$ varies with $\epsilon$. To pass to the limit, we construct a sequence $\{\phi^\epsilon\}\subset \mathcal{V}_{\mathcal{B}^{\epsilon T}}$ approximating any fixed $\phi\in \tilde{\mathcal{V}}_{\mathcal{B}^T}$.  

Following \cite[Section 5.2, p.~2055]{gerard2014existence}, introduce orthogonal curvilinear coordinates $(s_1,s_2,z)$ in a tubular neighborhood of $\partial \mathcal{B}^\epsilon_t$: $(s_1,s_2)$ parametrize $\partial \mathcal{B}^\epsilon_t$, and $z$ is the transverse signed distance, with $z>0$ outside $\mathcal{B}_t$. In particular, $\partial \mathcal{B}^\epsilon_t=\{z=0\}$. Let
\[
e_1=\tfrac{1}{h_1}\partial_{s_1},\qquad e_2=\tfrac{1}{h_2}\partial_{s_2},\qquad e_z=\tfrac{1}{h_z}\partial_{z},
\]
be the associated orthonormal frame, with scaling factors $h_1,h_2,h_z>0$.

Fix $\delta>0$ and take $\epsilon=\epsilon_\delta$ from Lemma \ref{lem:ConvergeBou1}. Define 
\begin{equation*}
\phi^\epsilon=\begin{cases}
\phi^F, & z>\delta,\\[2pt]
\langle\phi^F,e_1\rangle e_1+\langle\phi^F,e_2\rangle e_2+\tilde{\phi}_3\,e_z, & z\in(0,\delta),\\[2pt]
\phi^B, & z<0,
\end{cases}
\end{equation*}
where 
\[
\tilde{\phi}_3(s_1,s_2,z)=\eta\!\Big(\frac{3}{2\delta}z\Big)\langle\phi^B,e_z\rangle+\Big(1-\eta\!\Big(\frac{3}{2\delta}z\Big)\Big)\langle\phi^F,e_z\rangle,
\]
and $\eta\in C^\infty(\mathbb{R})$ satisfies
\[
\eta(\xi)=\begin{cases}
1, & |\xi|\le 1/2,\\
0, & |\xi|\ge 1.
\end{cases}
\]
Since $\phi \in \tilde{\mathcal{V}}_{\mathcal{B}^T}$, we have
$\phi^\epsilon \in \mathcal{V}_{\mathcal{B}^{\epsilon T}}$. Moreover, a direct computation shows that for $i=1,2$,
\begin{equation}\label{eq:JK6}
    \partial_{s_i}(\tilde{\phi}_3 e_z)
=\eta\!\Big(\tfrac{3}{2\delta}z\Big)\partial_{s_i}\!\big(\langle\phi^B,e_z\rangle e_z\big)
+\Big(1-\eta\!\Big(\tfrac{3}{2\delta}z\Big)\Big)\partial_{s_i}\!\big(\langle\phi^F,e_z\rangle e_z\big),
\end{equation}
and
\begin{equation}\label{eq:JK5}
    \begin{aligned}
        \partial_{z}(\tilde{\phi}_3 e_z)
&=\eta\!\Big(\tfrac{3}{2\delta}z\Big)\partial_{z}\!\big(\langle\phi^B,e_z\rangle e_z\big)
+\Big(1-\eta\!\Big(\tfrac{3}{2\delta}z\Big)\Big)\partial_{z}\!\big(\langle\phi^F,e_z\rangle e_z\big)
\\&+\frac{3}{2\delta}\,\eta'\!\Big(\tfrac{3}{2\delta}z\Big)\Big(\langle\phi^B,e_z\rangle-\langle\phi^F,e_z\rangle\Big)e_z.
    \end{aligned}
\end{equation}

For any point $A(s_1,s_2,z)$ with $z\in(0,\delta)$, define $\mathcal{A}^\epsilon_A\in\partial\mathcal{B}^\epsilon_t$ by
\begin{equation}\label{eq:Dec121}
    z=\mathrm{dist}(A,\partial\mathcal{B}^\epsilon_t)=\big|A\mathcal{A}^\epsilon_A\big|.
\end{equation}
By Lemma \ref{lem:ConvergeBou1}, there exists $\mathcal{A}_A\in \partial\mathcal{B}_t$ such that  
\begin{equation}\label{eq:Dec122}
    \big|\mathcal{A}_A^\epsilon\mathcal{A}_A\big|\leq\mathcal{O}(1)\delta^5, \qquad \big|n(\mathcal{A}_A^\epsilon)-n(\mathcal{A}_A)\big|\leq\mathcal{O}(1)\delta^5,
\end{equation}
where $n(\mathcal{A}_A^\epsilon)$ and $n(\mathcal{A}_A)$ denote the outward unit normal vectors to $\partial\mathcal{B}_t^\epsilon$ and $\partial\mathcal{B}_t$, respectively, at the corresponding points. Furthermore, combining \eqref{eq:Dec121} and \eqref{eq:Dec122}, we obtain that 
\begin{equation*}
    \big|A\mathcal{A}_A^\epsilon\big|\leq \mathcal{O}(1)\delta, \qquad\big|A\mathcal{A}_A\big|\leq \big|A\mathcal{A}_A^\epsilon\big|+\big|\mathcal{A}_A\mathcal{A}_A^\epsilon\big|\leq \mathcal{O}(1)\delta.
\end{equation*}

Then, noting the definition of $\tilde{\mathcal{V}}_{\mathcal{B}^T}$, direct computation shows that 
\begin{equation}\label{eq:JK1}
\begin{aligned}
&\big|\frac{3}{2\delta}\eta'\!\Big(\tfrac{3}{2\delta}z\Big)\Big(\langle\phi^B,e_z\rangle -\langle\phi^F,e_z\rangle\Big)e_z\big|\\
&=\Big|\frac{3}{2\delta}\eta'\!\Big(\tfrac{3}{2\delta}z\Big)\Big(\langle\phi^B(A),n(\mathcal{A}^\epsilon_A)\rangle-\langle\phi^B(A),n(\mathcal{A}_A)\rangle+\langle\phi^B(A),n(\mathcal{A}_A)\rangle-\langle\phi^B(\mathcal{A}_A),n(\mathcal{A}_A)\rangle\\
&+\langle\phi^F(\mathcal{A}_A),n(\mathcal{A}_A)\rangle-
\langle\phi^F(A),n(\mathcal{A}_A)\rangle
+\langle\phi^F(A),n(\mathcal{A}_A)\rangle-\langle\phi^F(A),n(\mathcal{A}^\epsilon_A)\rangle\Big)\Big|\\
&\leq\mathcal{O}(1)\delta^4+\mathcal{O}(1)\frac{1}{\delta} \qnt{\Big|\langle\phi^B(A)-\phi^B(\mathcal{A}_A),n(\mathcal{A}_A)\rangle-\langle\phi^F(A)-\phi^F(\mathcal{A}_A),n(\mathcal{A}_A)\rangle\Big|  }          
\end{aligned}
\end{equation}

When $\mathrm{dist}(A,\partial\mathcal{B}_t^\epsilon)\leq \delta^2$, we have
\[
|A\mathcal{A}_A| 
  \leq |A\mathcal{A}_A^\epsilon| + |\mathcal{A}_A^\epsilon \mathcal{A}_A|
  \leq \mathcal{O}(1)\,\delta^2.
\]
Since $\phi^B$ and $\phi^F$ possess high regularity, it follows from the Lagrange mean value theorem and \eqref{eq:JK1} that
\begin{equation}\label{eq:JK3}
\bigg|
\frac{3}{2\delta}\,\eta'\!\Big(\tfrac{3}{2\delta}z\Big)
\Big(\langle \phi^B, e_z\rangle - \langle \phi^F, e_z\rangle\Big)e_z
\bigg|
\leq \mathcal{O}(1)\,\delta,
\qquad z \in [0,\delta^2].
\end{equation}

When $\operatorname{dist}(A,\partial\mathcal{B}_t^\epsilon)\in (\delta^2,\delta)$, we need a delicate estimate of
\[
\Big|\langle\phi^B(A)-\phi^B(\mathcal{A}_A),n(\mathcal{A}_A)\rangle
      -\langle\phi^F(A)-\phi^F(\mathcal{A}_A),n(\mathcal{A}_A)\rangle\Big|
\]
appearing in \eqref{eq:JK1}. First, we note that
\begin{equation}\label{eq:normal-approx}
\Big|\tfrac{\mathcal{A}_A A}{|\mathcal{A}_A A|}-n(\mathcal{A}_A)\Big|
\;\le\; \mathcal{O}(1)\,\delta^3 .
\end{equation}
Denote $l_{\mathcal{A}_A A}:=\tfrac{\mathcal{A}_A A}{|\mathcal{A}_A A|}$. Since the normal components of $\phi^B$ and $\phi^F$ have $C^1$ contact across $\partial\mathcal{B}_t$, i.e.,
\[
(\phi^B-\phi^F)\cdot n=0,\qquad
n\cdot\nabla(\phi^B-\phi^F)n=0\quad\text{on } \partial\mathcal{B}_t,
\]
the Lagrange mean value theorem  yields
\begin{equation}\label{eq:JK2}
    \begin{aligned}
       &\Big|\langle\phi^B(A)-\phi^B(\mathcal{A}_A),n(\mathcal{A}_A)\rangle-\langle\phi^F(A)-\phi^F(\mathcal{A}_A),n(\mathcal{A}_A)\rangle\Big|\\
       &=\Big|\langle\partial_{l_{\mathcal{A}_AA}}\phi^B(A_1),n(\mathcal{A}_A)\rangle-\langle\partial_{l_{\mathcal{A}_AA}}\phi^F(A_2),n(\mathcal{A}_A)\rangle\Big|\Big|A\mathcal{A}_A\Big|\\
       &\leq \Big|\langle\partial_{n(\mathcal{A}_A)}\phi^B(A_1),n(\mathcal{A}_A)\rangle-\langle\partial_{n(\mathcal{A}_A)}\phi^F(A_2),n(\mathcal{A}_A)\rangle\Big|\Big|A\mathcal{A}_A\Big|+\mathcal{O}(1)\delta^3\\
       &\leq \Big|\langle\partial_{n(\mathcal{A}_A)}\phi^B(A_1),n(\mathcal{A}_A)\rangle-\langle\partial_{n(\mathcal{A}_A)}\phi^B(\mathcal{A}_A),n(\mathcal{A}_A)\rangle\\&~~~~+\langle\partial_{n(\mathcal{A}_A)}\phi^F(\mathcal{A}_A),n(\mathcal{A}_A)\rangle-\langle\partial_{n(\mathcal{A}_A)}\phi^F(A_2),n(\mathcal{A}_A)\rangle\Big|\Big|A\mathcal{A}_A\Big|+\mathcal{O}(1)\delta^3\\
       &\leq \Big|\langle\partial_{l_{\mathcal{A}_AA}n(\mathcal{A}_A)}\phi^B(A_1'),n(\mathcal{A}_A)\rangle\Big|\Big|A_1\mathcal{A}_A\Big|\Big|A\mathcal{A}_A\Big|\\&~~~~+\Big|\langle\partial_{l_{\mathcal{A}_AA}n(\mathcal{A}_A)}\phi^F(A_2'),n(\mathcal{A}_A)\rangle\Big|\Big|A_2\mathcal{A}_A\Big|\Big|A\mathcal{A}_A\Big|+\mathcal{O}(1)\delta^3\\&\leq\mathcal{O}(1)\delta^2,
    \end{aligned}
\end{equation}
where $A_i, A'_i, i=1,2$ are points on the segment $\mathcal{A}_AA.$ Then, plugging \eqref{eq:JK2} into \eqref{eq:JK1}, we have that 
\begin{equation}\label{eq:JK4}
    \begin{aligned}
        \big|\frac{3}{2\delta}\eta'\!\Big(\tfrac{3}{2\delta}z\Big)\Big(\langle\phi^B,e_z\rangle -\langle\phi^F,e_z\rangle\Big)e_z\big|\leq \mathcal{O}(1)\,\delta, \qquad z\in(\delta^2,\delta]. 
    \end{aligned}
\end{equation}

Finally, plugging \eqref{eq:JK3} and \eqref{eq:JK4} into \eqref{eq:JK5}, and recalling \eqref{eq:JK6}, we conclude the property of the approximation of the test function in the following lemma:  
\begin{lemma}\label{lem:JJ1}
 The approximated test function satisfies  
   \begin{equation}
       \begin{aligned}
           \phi^{\epsilon}&=\phi^F\chi_{\{z\ge0\}}+\phi^B\chi_{\{z<0\}}+\eta\!\Big(\frac{3}{2\delta}z\Big)\langle\phi^B-\phi^F,e_z\rangle e_z \chi_{\{0\le z\le\delta\}},\\&=\phi^F\chi_{\{z\ge0\}}+\phi^B\chi_{\{z<0\}}+\mathcal{O}(\delta),
       \end{aligned}
   \end{equation}
   and 
\begin{equation}
    \begin{aligned}
    \nabla\phi^\epsilon &=\nabla\phi^F\chi_{\{z\ge0\}}+\nabla\phi^B\chi_{\{z<0\}}+\eta\!\Big(\frac{3}{2\delta}z\Big)\nabla\qnt{\langle\phi^B-\phi^F,e_z\rangle e_z} \chi_{\{0\le z\le\delta\}}+\mathcal{O}(1)\delta\\&=\nabla\phi^F\chi_{\{z\ge0\}}+\nabla\phi^B\chi_{\{z<0\}}+\mathcal{O}(\delta).    
    \end{aligned}
\end{equation}
\end{lemma}
\begin{remark}
The last equalities in both formulas hold due to the $C^1$ regularity of the normal component of $\phi\in\tilde{\mathcal{V}}_{\mathcal{B}^T}$ at $\partial\mathcal{B}_t$. 
\end{remark}

\smallskip
\textbf{Step 3: The construction of $(\nu,\mathcal{D},\mathcal{B}_t)$.} %Set 
%\[
%\mathcal{F}^{\epsilon,\delta}_t=\{x\in\mathcal{F}_t^\epsilon:\mathrm{dist}(x,\partial\mathcal{F}_t^\epsilon)>\delta\}, \qquad 
%d\mathcal{F}^{\epsilon,\delta}_t=\mathcal{F}^\epsilon_t\setminus\mathcal{F}^{\epsilon,\delta}_t.
%\]
From \eqref{eq:EnergyInequalityofVisSys},
\[
\chi_{\mathcal{F}^{\epsilon}_t}\rho^\epsilon\in L^\infty(0,T;L^\gamma(\Omega)),\qquad 
\chi_{\mathcal{F}^{\epsilon}_t}\sqrt{\rho^\epsilon}\,|u^\epsilon|\in L^\infty(0,T;L^2(\Omega)).
\]
By Lemma \ref{lemma:YoungMeasure1}, there exist 
\[
\mathcal{Y}\in L^{\infty}_w\big((0,T)\times\Omega;\mathcal{P}(\mathbb{R}^{1+3})\big),\qquad 
m_t^{(\cdot)}: C_{\gamma,2}\to \mathcal{M}(\overline{\Omega})
\]
such that for any $f\in C_{\gamma,2}$,
\begin{equation}\label{eq:limitofApp}
f(\chi_{\mathcal{F}^{\epsilon}_t}\rho^\epsilon,\chi_{\mathcal{F}^{\epsilon}_t}\sqrt{\rho^\epsilon}u^\epsilon)\,\mathcal{L}
\overset{*}{\rightharpoonup}\int_{\mathbb{R}^{1+3}} f(\lambda_1,\lambda')\,\mathrm{d}\mathcal{Y}\,\mathcal{L}
+ m^f_t\otimes \mathrm{d}t 
\end{equation}
as $\delta$ and $\epsilon=\epsilon_\delta$ tend to zero.
Moreover, if $f(0,0)=0$, the limiting measure in \eqref{eq:limitofApp} is supported in $\overline{\mathcal{F}^T}$.

Let $\phi\in\tilde{\mathcal{V}}_{\mathcal{B}^T}$ and $\phi^\epsilon$ be the associated approximation. Then, for $f\in C_{\gamma,2}$ with $f(0,0)=0$, 
%\begin{equation}\label{eq:GeneralConvergence}
%\begin{aligned}
%\int_0^T\!\!\int_{\mathcal{F}^\epsilon_\tau} f(\rho^\epsilon,\sqrt{\rho^\epsilon}u^\epsilon)\,\nabla\phi^\epsilon\,\mathrm{d}x\,\mathrm{d}\tau
%&=\int_0^T\!\!\int_{\Omega} f(\chi_{\mathcal{F}^{\epsilon,\delta}_\tau}\rho^\epsilon,\chi_{\mathcal{F}^{\epsilon,\delta}_\tau}\sqrt{\rho^\epsilon}u^\epsilon)\,\nabla\phi^\epsilon\,\mathrm{d}x\,\mathrm{d}\tau\\
%&\quad+ \int_0^T\!\!\int_{\Omega} f(\chi_{d\mathcal{F}^{\epsilon,\delta}_\tau}\rho^\epsilon,\chi_{d\mathcal{F}^{\epsilon,\delta}_\tau}\sqrt{\rho^\epsilon}u^\epsilon)\,\nabla\phi^\epsilon\,\mathrm{d}x\,\mathrm{d}\tau.
%\end{aligned}
%\end{equation} 
according to Lemma \ref{lem:ConvergeBou1}, \eqref{eq:limitofApp}, and the properties of $\phi^\epsilon$ in Lemma \ref{lem:JJ1}, we obtain
\begin{equation}\label{eq:momentumConvergence1}
\begin{gathered}
\int_0^T\int_{\Omega} f(\chi_{\mathcal{F}^{\epsilon}_\tau}\rho^\epsilon,\chi_{\mathcal{F}^{\epsilon}_\tau}\sqrt{\rho^\epsilon}u^\epsilon)\,\nabla\phi^\epsilon\,\mathrm{d}x\,\mathrm{d}\tau \\[6pt]
\downarrow\ (\epsilon\to0) \\[6pt]
\int_0^T\int_{\mathcal{F}_\tau}\!\langle f(\lambda_1,\lambda'),\mathcal{Y}\rangle\,\nabla\phi^F\,\mathrm{d}x\,\mathrm{d}\tau
+\int_0^T\int_{\overline{\mathcal{F}_\tau}}\!\nabla\phi^F\,\mathrm{d}m^f_\tau\,\mathrm{d}\tau.
\end{gathered}
\end{equation}
%and 
%\[
%\begin{gathered}
%\int_0^T\int_{\Omega} f(\chi_{d\mathcal{F}^{\epsilon,\delta}_\tau}\rho^\epsilon,\chi_{d\mathcal{F}^{\epsilon,\delta}_\tau}\sqrt{\rho^\epsilon}u^\epsilon)\,\nabla\phi^\epsilon\,\mathrm{d}x\,\mathrm{d}\tau \\[6pt]
%\downarrow\ (\epsilon\to0) \\[6pt]
%\int_0^T\int_{\Omega}\!\nabla\phi^F\,\mathrm{d}\tilde{m}^f_\tau\,\mathrm{d}\tau,
%\end{gathered}
%\]
%for some $\tilde{m}^f_t$ supported on $\partial\mathcal{F}_t$. 
%Therefore,
%\begin{equation}\label{eq:momentumConvergence1}
%\begin{gathered}
%\int_0^T\int_{\mathcal{F}^\epsilon_\tau} f(\rho^\epsilon,\sqrt{\rho^\epsilon}u^\epsilon)\,\nabla\phi^\epsilon\,\mathrm{d}x\,\mathrm{d}\tau \\[6pt]
%\downarrow\ (\epsilon\to0) \\[6pt]
%\int_0^T\int_{\mathcal{F}_\tau}\!\langle f(\lambda_1,\lambda'),\mathcal{Y}\rangle\,\nabla\phi^F\,\mathrm{d}x\,\mathrm{d}\tau
%+\int_0^T\int_{\overline{\mathcal{F}_\tau}}\!\nabla\phi^F\,\mathrm{d}\nu^f_\tau\,\mathrm{d}\tau,
%\end{gathered}
%\end{equation}
%where 
%\begin{equation}
 %   \nu^f_t:=m^f_t+\tilde{m}^f_t.
%\end{equation}

We now set
\[
    \nu=\begin{cases}
        \mathcal{Y},& (t,x)\in\mathcal{F}^T,\\
        \delta_{(\rho_B,\sqrt{\rho_B}u_B)},& (t,x)\in \mathcal{B}^T,
    \end{cases}
\]
and 
\begin{equation}\label{eq:Defmeasuers}
   \nu^M= m^{\lambda'\otimes\lambda'+\lambda_1^\gamma}_t,\qquad 
  \mathcal{D}(t)=m^{\frac{|\lambda'|^2}{2}+\frac{\lambda_1^\gamma}{\gamma-1}}_t(\Omega),
\end{equation}
with $\mathcal{B}^T$ and $(\rho_B, u_B)$ as in Step~1.

\smallskip
\textbf{Step 4: Verification of constructed $\nu, \mathcal{D}, \mathcal{B}^T$ as a solution.} 
From 
\[
\sup_{\epsilon>0}\|\chi_{\mathcal{F}^\epsilon_t}\rho^\epsilon\|_{L^\infty(0,T;L^{\gamma}(\Omega))}
+\sup_{\epsilon>0}\|\chi_{\mathcal{F}^\epsilon_t}\sqrt{\rho^\epsilon}u^\epsilon\|_{L^\infty(0,T;L^{2}(\Omega))}<+\infty,
\]
we obtain 
\[
\sup_{\epsilon>0}\|\chi_{\mathcal{F}^\epsilon_t}\rho^\epsilon u^\epsilon\|_{L^\infty(0,T;L^{\frac{2\gamma}{\gamma+1}}(\Omega))}<+\infty.
\]
Using these bounds together with Lemma \ref{lem:ConvergeBou1}, \eqref{eq:ConvergeBou}, \eqref{eq:SolidConvergence1}, \eqref{eq:SolidConvergence111}, and passing to the limit in \eqref{eq:DefofContinuityEquation}, we conclude that $\nu,\mathcal{D},\mathcal{B}^T$ satisfy items (1)–(3) of Definition~\ref{def:definitionofYoungmeasure}. 

For the momentum balance, we estimate
\[
\begin{aligned}
&\Big|\epsilon\!\int_0^t\!\!\int_{\mathcal{F}^\epsilon_\tau }\mathcal{T}(u^\epsilon): \nabla \phi \,\mathrm{d}x\,\mathrm{d}\tau \Big|
\le \sqrt{\epsilon}\,\|\sqrt{\epsilon}\,\mathcal{T}(u^\epsilon)\|_{L^2(\mathcal{F}^{\epsilon T})}\,\|\nabla\phi\|_{L^2(\mathcal{F}^{\epsilon T})},\\&
\Big|\epsilon \!\int_0^t \!\!\int_{\partial \Omega} (u^\epsilon_F \times n) \cdot (\phi^F \times n)\,\mathrm{d}S\,\mathrm{d}\tau\Big|
\le\sqrt{\epsilon}\,\|\sqrt{\epsilon}\,u^\epsilon_F\times  n\|_{L^2(0,T;L^2(\partial\Omega))}\,\|\phi^F\times n\|_{L^2(0,T;L^2(\partial\Omega))},\\&
\Big|\epsilon \!\int_0^t \!\!\int_{\partial \mathcal{B}^\epsilon_\tau } \big[(u^\epsilon_F - u^\epsilon_B) \times n\big] \cdot \big[(\phi^F - \phi^B) \times n\big] \,\mathrm{d}S\,\mathrm{d}\tau\Big|\\&
\le\sqrt{\epsilon}\,\|\sqrt{\epsilon}(u^\epsilon_F-u^\epsilon_B)\times n\|_{L^2(0,T;L^2(\partial\mathcal{B}^\epsilon_\tau ))}\,\|(\phi^F-\phi^B)\times n\|_{L^2(0,T;L^2(\partial\mathcal{B}^\epsilon_\tau ))}.
\end{aligned}
\]
Together with \eqref{eq:EnergyInequalityofVisSys}, this gives 
\begin{equation}\label{eq:Demandforfinalresults}
    \begin{aligned}
        &\epsilon\int_0^t \int_{\mathcal{F}^\epsilon_\tau }\mathcal{T}(u^\epsilon): \nabla \phi \mathrm{d}x\mathrm{d}\tau \xrightarrow{\epsilon\to0} 0,\\
    &\epsilon \int_0^t  \int_{\partial \Omega} (u^\epsilon_F \times n) \cdot (\phi^F \times n)\xrightarrow{\epsilon\to0} 0,\\& \epsilon \int_0^t  \int_{\partial \mathcal{B}^\epsilon_\tau } [(u^\epsilon_F - u^\epsilon_B) \times n] \cdot [(\phi^F - \phi^B) \times n] \xrightarrow{\epsilon\to0} 0.
    \end{aligned}
\end{equation}
 Hence, by \eqref{eq:momentumConvergence1} and the rigid-body convergences, letting $\epsilon\to0$ in \eqref{eq:ViscousMomentumEstimates} yields that $\nu,\mathcal{D},\mathcal{B}^T,\nu^M$ satisfy item (4) (balance of momentum) in Definition~\ref{def:definitionofYoungmeasure}. 

Finally, by the definition of $\mathcal{D}(t)$ in \eqref{eq:Defmeasuers}, $\nu,\mathcal{D},\mathcal{B}^T$ satisfy item (5) (dissipation of kinetic energy). This completes the proof of Theorem \ref{thm:ExistenceofYMVS}. 
\end{proof}

\section{Weak-Strong Uniqueness}

In this section, we prove that the constructed dissipative Young measure-valued solution satisfies the weak-strong uniqueness property.  
Denote the dissipative Young measure-valued solution obtained in Section~2 by
\[
(\nu=\chi_{\overline{\mathcal{F}_{1,t}}}\nu
   +\chi_{\mathcal{B}_{1,t}}\delta_{(\rho_{B1},\sqrt{\rho_{B1}}u_{B1})},
   \mathcal{D}(t), \mathcal{B}_{1,t})_{t\in[0,T]},
\]
and denote the strong solution to the fluid-rigid system by
\[
(\rho_{F2},u_{F2},\rho_{B2},u_{B2},\mathcal{B}_{2,t})_{t\in[0,T]},
\]
where
\begin{equation}
    u_{B i}=V_i(t)+\omega_i(t)\times(x-X_i(t)),\qquad i=1,2.
\end{equation}
We further assume that the strong solution and the dissipative weak solution in the section satisfy the condition in Theorem \ref{thm:Weak-strongUniqueness}.

\subsection{Change of Coordinates}

Define
\[
\Omega_\sigma:=\{x\in\Omega:\mathrm{dist}(x,\partial\Omega)>\sigma\},\qquad
\Omega_{\frac{\sigma}{2}}:=\{x\in\Omega:\mathrm{dist}(x,\partial\Omega)>\tfrac{\sigma}{2}\}.
\]
The assumption $\max_{i=1,2}\sup_{t\in[0,T]}
  \mathrm{dist}(\partial\Omega,\partial\mathcal{B}_{i,t})
  >\frac{3\sigma}{2}$ implies that
\[
\mathcal{B}_{i,t}\subset\Omega_\sigma,\qquad i=1,2,\ t\in(0,T).
\]

Let $\xi(x)\in C_c^\infty(\Omega)$ be a smooth truncation such that $\xi=1$ in $\Omega_\sigma$ and $\xi=0$ in $\Omega\setminus\Omega_{\frac{\sigma}{2}}$. Referring to \cite[Lemma 1]{MR2888201}, we define the truncated velocity fields
\begin{equation}
    \Lambda_i(t,x)=-\frac{1}{2}\text{curl}(\xi(x)(x\times V_i(t)+|x-X_i(t)|^2\omega_i(t))),\qquad i=1,2,
\end{equation}
which satisfy that 
\begin{equation}
    \operatorname{div}(\Lambda_i)=0, ~~ \Lambda_i\big|_{\Omega_\sigma}=u_{Bi}, ~~ \Lambda_i\big|_{\Omega\setminus\Omega_{{\sigma}/{2}}}=0, \quad i=1,2. 
\end{equation}

Moreover, according to \eqref{eq:BoundsonDesity} and \eqref{eq:SolidConvergence111}, $\rho_{B1}$ admits a positive lower bound.  
Together with the energy dissipation inequality, this yields $u_{B1}\in L^\infty(0,T;\mathcal{R}(\Omega))$, i.e.,
\[
\sup_{t\in(0,T)}(|V_1(t)|+|\omega_1(t)|)<+\infty.
\]
Consequently,
\begin{equation}
    \Lambda_1\in L^\infty(0,T;C^\infty(\Omega)),\qquad
    \Lambda_2\in C^1(0,T;C^\infty(\Omega)).
\end{equation}

We now define the flow maps
\[
Z_i(t,\cdot):\Omega\to\Omega,\qquad x\mapsto Z_i(t,x),\quad i=1,2,
\]
as the unique solutions of
\begin{equation}
\begin{cases}
\dfrac{\mathrm{d}}{\mathrm{d}t}Z_i(t,x)=\Lambda_i(t,Z_i(t,x)),\\[6pt]
Z_i(0,x)=x.
\end{cases}
\end{equation}

The following lemma collects useful properties of these coordinate transformations.

\begin{lemma}\label{lem:PropertiesZi}
Define
\[
\tilde{Z}_1(t,\cdot):=Z_1(t,\cdot)\circ Z_2^{-1}(t,\cdot),\qquad
\tilde{Z}_2(t,\cdot):=Z_2(t,\cdot)\circ Z_1^{-1}(t,\cdot).
\]
Then, for each fixed $t$, the following properties hold:
\begin{equation}
\begin{aligned}
&\tilde{Z}_1\circ\tilde{Z}_2=\mathbb{I},\\
&\tilde{Z}_1(t,\cdot) \text{~and~} \tilde{Z}_2(t,\cdot) \text{~are volume-preserving maps,}\\
&\tilde{Z}_1(t,x)=X_1(t)+\mathbb{O}_{1,t}\mathbb{O}_{2,t}^\top(x-X_2(t)),\qquad \text{~in a neighborhood of~}\mathcal{B}_{2,t},\\
&\tilde{Z}_2(t,x)=X_2(t)+\mathbb{O}_{2,t}\mathbb{O}_{1,t}^\top(x-X_1(t)),\qquad \text{~in a neighborhood of~}\mathcal{B}_{1,t},
\end{aligned}
\end{equation}
where $\mathbb{O}_{i,t}$ are the rotation matrices associated with the rigid motions of $\mathcal{B}_{i,t}$, defined by \eqref{eq:rotation-ode}.  
Moreover,
\begin{equation}
\tilde{Z}_2(t,\mathcal{B}_{1,t})=\mathcal{B}_{2,t},\qquad t\in(0,T).
\end{equation}
\end{lemma}

Lemma~\ref{lem:PropertiesZi} allows us to redefine $\rho_{F2},u_{F2}$ and $\rho_{B2},u_{B2}$ over the regions $\cup_{t\in(0,T)}\{t\}\times\mathcal{F}_{1,t}$ and $\cup_{t\in(0,T)}\{t\}\times\mathcal{B}_{1,t}$, respectively.  
Introducing the change of variables
\begin{equation}
\tau=t,\qquad y=\tilde{Z}_2(t,x),
\end{equation}
we set
\begin{equation}\label{eq:TransformFluidp}
U_2(t,x)=\frac{\partial x}{\partial y}(\tau,y)\,u_{F2}(\tau,y),\qquad
R_2(t,x)=\rho_{F2}(\tau,y),\qquad
P_2(t,x)=p_{F2}(\tau,y),
\end{equation}
and
\begin{equation}\label{eq:TransformedVelocity}
\begin{aligned}
&\tilde{V}_2(t)=\mathbb{O}_{1,t}\mathbb{O}_{2,t}^\top V_2(t), \qquad 
\tilde{\omega}_2(t)=\mathbb{O}_{1,t}\mathbb{O}_{2,t}^\top \omega_2(t),\\
&\tilde{u}_{B2}(t,x)=\tilde{V}_2(t)+\tilde{\omega}_2(t)\times(x-X_1(t)),\qquad
\tilde{\rho}_{B2}(t,x)=\rho_{B2}(\tau,y).
\end{aligned}
\end{equation}

Using the Einstein summation convention, we note
\[
\partial_\tau=\partial_t+\frac{\partial x_i}{\partial \tau}\partial_{x_i},\qquad
\partial_{y_i}=\frac{\partial x_j}{\partial y_i}\partial_{x_j},
\]
and
\[
u_{F2}^i(\tau,y)=\frac{\partial y_i}{\partial x_l}(t,x)\,U_2^l(t,x).
\]

Moreover, recalling the Jacobi's formula $$\partial_{x_l}\operatorname{det}(\frac{\partial y}{\partial x})=\operatorname{det}(\frac{\partial y}{\partial x})\operatorname{tr}\Big(\frac{\partial x}{\partial y}\cdot\partial_{x_l}\frac{\partial y}{\partial x}\Big),$$ 
and the fact that $ \tilde{Z}_2(t,\cdot)$ is volume-preserving maps,  we have that 
$$\operatorname{tr}\Big(\frac{\partial x}{\partial y}\cdot\partial_{x_l}\frac{\partial y}{\partial x}\Big)=\frac{\partial x_j}{\partial y_i}\partial_{x_l}(\frac{\partial y_i}{\partial x_j})=0.$$
Here $\operatorname{det}(\cdot)$ and $\operatorname{tr}(\cdot)$ denote the determinant and the trace of a matrix, respectively.

Now, we are ready to derive the governing equations for $R_2,U_2,P_2,\tilde{u}_{B2}$.  To this end, a direct computation yields
\begin{equation}\label{eq:MassConserv}
\begin{aligned}
0&=\partial_\tau\rho_{F2}(\tau,y)+\partial_{y_i}(\rho_{F2}u_{F2}^i)\\
&=\partial_tR_2+\frac{\partial x_i}{\partial\tau}\partial_{x_i}R_2
+\frac{\partial x_j}{\partial y_i}\partial_{x_j}(R_2U_2^l\frac{\partial y_i}{\partial x_l})\\
&=\partial_tR_2+\partial_{x_i}(R_2U_2^i)+\frac{\partial x_i}{\partial\tau}\partial_{x_i}R_2+\frac{\partial x_j}{\partial y_i}\partial_{x_j}(\frac{\partial y_i}{\partial x_l})R_2U_2^l\\&=\partial_tR_2+\partial_{x_i}(R_2U_2^i)+\frac{\partial x_i}{\partial\tau}\partial_{x_i}R_2+\frac{\partial x_j}{\partial y_i}\partial_{x_l}(\frac{\partial y_i}{\partial x_j})R_2U_2^l\\&=\partial_tR_2+\partial_{x_i}(R_2U_2^i)+\frac{\partial x_i}{\partial\tau}\partial_{x_i}R_2.
\end{aligned}
\end{equation}

Similarly,
\begin{equation}
\begin{aligned}
0&=\partial_\tau(\rho_{F2}u_{F2}^i)+\partial_{y_l}(\rho_{F2}u_{F2}^iu_{F2}^l+p_{F2}\delta^{il})\\
&=\frac{\partial y_i}{\partial x_l}\partial_t(R_2U_2^l)
+\partial_t\!\Big(\frac{\partial y_i}{\partial x_l}\Big)R_2U_2^l
+\frac{\partial x_j}{\partial\tau}\frac{\partial y_i}{\partial x_l}\partial_{x_j}(R_2U_2^l)
+\frac{\partial x_j}{\partial\tau}\partial_{x_j}\!\Big(\frac{\partial y_i}{\partial x_l}\Big)R_2U_2^l\\& \quad+\frac{\partial y_i}{\partial x_m}\partial_{x_j}\big(R_2U_2^mU_2^j\big)+\partial_{x_j}\Big(\frac{\partial y_i}{\partial x_m}\Big)R_2U_2^mU_2^j+\frac{\partial x_j}{\partial y_l}\frac{\partial y_i}{\partial x_m}\partial_{x_j}\Big(\frac{\partial y_l}{\partial x_n}\Big)R_2 U_2^m U_2^n\\&\quad+\frac{\partial x_j}{\partial y_l}\partial_{x_j}\Big(P_2\delta^{il}\Big),
\end{aligned}
\end{equation}
which simplifies to
\begin{equation}\label{MometumConserv}
\begin{aligned}
0&=\partial_t(R_2U_2^l)+\partial_{x_j}(R_2U_2^lU_2^j+P_2\delta^{lj})\\
&\quad+\frac{\partial x_l}{\partial y_i}\!\Big(
\partial_t\frac{\partial y_i}{\partial x_k}
+\frac{\partial x_j}{\partial\tau}\partial_{x_j}\frac{\partial y_i}{\partial x_k}\Big)
R_2U_2^k
+\frac{\partial x_j}{\partial\tau}\partial_{x_j}(R_2U_2^l)\\&\quad+\frac{\partial x_l}{\partial y_i}\partial_{x_j}\Big(\frac{\partial y_i}{\partial x_m}\Big)R_2U_2^mU_2^j+\Big(\frac{\partial x_l}{\partial y_i}\frac{\partial x_j}{\partial y_i}-\delta^{lj}\Big)\partial_{x_j}P_2
\end{aligned}
\end{equation}

The rigid-body equations transform into
\begin{equation}\label{eq:Rigidmotionconservation}
\begin{aligned}
m\frac{\mathrm{d}\tilde{V}_2(t)}{\mathrm{d}t}
&=m(\tilde{\omega}_2-\omega_1)\times\tilde{V}_2
+\int_{\partial\mathcal{B}_1(t)}P_2\,n\,\mathrm{d}S,\\
\mathcal{J}_1\frac{\mathrm{d}\tilde{\omega}_2(t)}{\mathrm{d}t}
&=\mathcal{J}_1\tilde{\omega}_2\times\tilde{\omega}_2
-(\tilde{\omega}_2-\omega_1)\times(\mathcal{J}_1\tilde{\omega}_2)
+\int_{\partial\mathcal{B}_1(t)}(x-X_1(t))\times P_2\,n\,\mathrm{d}S,
\end{aligned}
\end{equation}
where $\mathcal{J}_1$ and $\mathcal{J}_2$ are the inertia tensors of the two configurations and satisfy
\[
\mathcal{J}_1=\mathbb{O}_{1,t}\mathbb{O}_{2,t}^\top\mathcal{J}_2\mathbb{O}_{2,t}\mathbb{O}_{1,t}^\top.
\]
The corresponding boundary conditions are
\begin{equation}\label{eq:Bound}
\begin{aligned}
U_2\cdot n&=\tilde{u}_{B2}\cdot n\quad\text{on }(0,T)\times\partial\mathcal{B}_{1,t},\\
U_2\cdot n&=0\quad\text{on }(0,T)\times\partial\Omega.
\end{aligned}
\end{equation}

Finally, we record the following estimates for $\tilde{Z}_i(t,x)$, $i=1,2$ (see \cite[Lemmas~3.1–3.2]{caggio2021measure} and \cite[Corollary~1]{MR3375542}).

\begin{lemma}\label{PropertiesonTransform}
For $t\in(0,T)$, it holds that
\begin{equation}
\begin{aligned}
\|\tilde{Z}_i(t,\cdot)-\mathrm{Id}\|_{W^{3,\infty}(\mathcal{F}_t)}
&\le \mathcal{O}(1)\big(\|V_1-\tilde{V}_2\|_{L^2(0,T)}
+\|\omega_1-\tilde{\omega}_2\|_{L^2(0,T)}\big),\\
\|\partial_t\tilde{Z}_i(t,\cdot)\|_{W^{1,\infty}(\mathcal{F}_t)}
&\le \mathcal{O}(1)\big(|V_1(t)-\tilde{V}_2(t)|
+|\omega_1(t)-\tilde{\omega}_2(t)|\big).
\end{aligned}
\end{equation}
\end{lemma}

\subsection{Proof of Theorem \ref{thm:Weak-strongUniqueness} \,(Relative energy analysis)}

In this subsection, for clarity, we denote the measure-valued solution by \((\nu,\mathcal{D}(t),\mathcal{B}_t)_{t\in(0,T)}\) and the redefined (transformed) strong solution by \((R,U,\tilde V(t),\tilde\omega(t),\tilde u_B,\tilde\rho_B, \mathcal{B}_t)\).

According to \eqref{eq:MassConserv}--\eqref{eq:Bound}, it holds that
\begin{equation}\label{eq:transormedStrongSolution}
\left\{
\begin{aligned}
&\partial_t R + \partial_{x_i}(R\,U^i)
     \;=\; -\frac{\partial x_i}{\partial \tau}\,\partial_{x_i}R \;=:\; \mathcal{H},\\
&\partial_t(R\,U^l) + \partial_{x_j}(R\,U^lU^j + P\,\delta^{lj}) \\
&\quad \;=\; -\frac{\partial x_l}{\partial y_i}\!\Big(
\partial_t\frac{\partial y_i}{\partial x_k}
+\frac{\partial x_j}{\partial\tau}\partial_{x_j}\frac{\partial y_i}{\partial x_k}\Big)
R_2U_2^k
-\frac{\partial x_j}{\partial\tau}\partial_{x_j}(R_2U_2^l)\\&\qquad-\frac{\partial x_l}{\partial y_i}\partial_{x_j}\Big(\frac{\partial y_i}{\partial x_m}\Big)R_2U_2^mU_2^j-\Big(\frac{\partial x_l}{\partial y_i}\frac{\partial x_j}{\partial y_i}-\delta^{lj}\Big)\partial_{x_j}P_2
      \;=:\; \mathcal{G}^l,
      \qquad (t,x)\in \mathcal{F}^T,\\
&U\cdot n=\tilde u_B\cdot n\quad \text{on } (0,T)\times\partial\mathcal{B}_t,\\
&U\cdot n=0\quad \text{on } (0,T)\times\partial\Omega,\\
&m\,\frac{\mathrm{d}\tilde V(t)}{\mathrm{d}t}
   = m(\tilde\omega-\omega)\times \tilde V
     + \int_{\partial\mathcal{B}(t)} P\,n\,\mathrm{d}S,\\
&\mathcal{J}\,\frac{\mathrm{d}\tilde\omega(t)}{\mathrm{d}t}
   = \mathcal{J}\tilde\omega\times\tilde\omega
     -(\tilde\omega-\omega)\times (\mathcal{J}\tilde\omega)
     + \int_{\partial\mathcal{B}(t)} (x-X(t))\times P\,n\,\mathrm{d}S,
\qquad t\in(0,T).
\end{aligned}
\right.
\end{equation}

Following \cite{gwiazda2015weak,MR3567640}, define the relative energy
\begin{equation}\label{eq:defofRE}
\begin{aligned}
E_{\mathrm{rel}}(t)
&=\int_{\mathcal{F}_t}\frac{1}{2}\Big\langle\big|\lambda'-\sqrt{\lambda_1}\,U\big|^2,\nu\Big\rangle\,\mathrm{d}x
 +\int_{\mathcal{F}_t}\Big\langle \frac{\lambda_1^\gamma}{\gamma-1}
     -\frac{\gamma R^{\gamma-1}\lambda_1}{\gamma-1}+R^\gamma,\nu\Big\rangle\,\mathrm{d}x\\
&\quad + \int_{\mathcal{B}_t}\frac{1}{2}\,\rho_B\,|u_B-\tilde u_B|^2\,\mathrm{d}x,
\end{aligned}
\end{equation}
which expands to
\begin{equation}\label{ExpandedRelativeEnergy}
\begin{aligned}
E_{\mathrm{rel}}(t)
&=\int_{\mathcal{F}_t}\Big\langle \frac{|\lambda'|^2}{2}+\frac{\lambda_1^\gamma}{\gamma-1},\nu\Big\rangle\,\mathrm{d}x
  +\frac{1}{2}\int_{\mathcal{B}_t}\rho_B|u_B|^2\,\mathrm{d}x\\
&\quad+\int_{\mathcal{F}_t}\Big(\frac{1}{2}\langle \lambda_1,\nu\rangle |U|^2+R^\gamma\Big)\,\mathrm{d}x
     +\frac{1}{2}\int_{\mathcal{B}_t}\rho_B|\tilde u_B|^2\,\mathrm{d}x\\
&\quad-\int_{\mathcal{F}_t}\Big\langle \sqrt{\lambda_1}\lambda',\nu\Big\rangle\!\cdot U
        +\frac{\gamma}{\gamma-1}\langle \lambda_1,\nu\rangle R^{\gamma-1}\,\mathrm{d}x
      -\int_{\mathcal{B}_t}\rho_B\,u_B\cdot \tilde u_B\,\mathrm{d}x\\
&=E_{\mathrm{mvs}}(t)
  +\int_{\mathcal{F}_t}\Big(\frac{1}{2}\langle \lambda_1,\nu\rangle |U|^2+R^\gamma\Big)\,\mathrm{d}x
  +\frac{1}{2}\int_{\mathcal{B}_t}\rho_B|\tilde u_B|^2\,\mathrm{d}x\\
&\quad-\int_{\mathcal{F}_t}\Big\langle \sqrt{\lambda_1}\lambda',\nu\Big\rangle\!\cdot U
      +\frac{\gamma}{\gamma-1}\langle \lambda_1,\nu\rangle R^{\gamma-1}\,\mathrm{d}x
      -\int_{\mathcal{B}_t}\rho_B\,u_B\cdot \tilde u_B\,\mathrm{d}x,
\end{aligned}
\end{equation}
where \(E_{\mathrm{mvs}}(t)\) is defined in \eqref{eq:EnergyMVS}.

Taking \(\Psi^F=U\), \(\Psi^B=\tilde u_B\) in \eqref{MVS MomentumC} yields
\begin{equation}\label{testby U}
\begin{aligned}
&\int_{\mathcal{F}_t}\Big\langle\sqrt{\lambda_1}\lambda',\nu\Big\rangle\!\cdot U(t,\cdot)\,\mathrm{d}x
+\int_{\mathcal{B}_t}\rho_B u_B\cdot \tilde u_B\,\mathrm{d}x
\\&=\int_0^t\!\!\int_{\mathcal{F}_\tau}
  \Big\langle\sqrt{\lambda_1}\lambda',\nu\Big\rangle\!\cdot \partial_\tau U
 +\Big\langle\lambda'\!\otimes\!\lambda'+\lambda_1^\gamma \mathbb{I},\nu\Big\rangle\!:\!\nabla U\,\mathrm{d}x\,\mathrm{d}\tau\\
&\quad+\int_0^t\!\!\int_{\mathcal{B}_\tau}\rho_B u_B\cdot \partial_\tau \tilde u_B\,\mathrm{d}x\,\mathrm{d}\tau
+\int_0^t\!\!\int_{\overline{\mathcal{F}_\tau}}\nabla U\,\mathrm{d}\nu^M\,\mathrm{d}\tau\\
&\quad+\int_{\mathcal{F}_0}R_0|U_0|^2\,\mathrm{d}x
+\int_{\mathcal{B}_0}\rho_{B0}|u_{B0}|^2\,\mathrm{d}x.
\end{aligned}
\end{equation}

Similarly, taking \(\psi|_{\mathcal{F}^T}=\frac{1}{2}|U|^2\) and \(\psi|_{\mathcal{F}^T}=\gamma R^{\gamma-1}\) respectively in \eqref{MVS MassC}, we obtain
\begin{equation}\label{textby UU}
\begin{aligned}
&\frac{1}{2}\int_{\mathcal{F}_t}\langle \lambda_1,\nu\rangle |U|^2\,\mathrm{d}x
\\&=\int_0^t\!\!\int_{\mathcal{F}_\tau}
 U\cdot \partial_\tau U\,\langle \lambda_1,\nu\rangle
 +\Big\langle \sqrt{\lambda_1}\lambda',\nu\Big\rangle\!\cdot(\nabla U\,U)\,\mathrm{d}x\,\mathrm{d}\tau
 +\frac{1}{2}\int_{\mathcal{F}_0}R_0|U_0|^2\,\mathrm{d}x,
\end{aligned}
\end{equation}
and
\begin{equation}\label{testby R power}
\begin{aligned}
&\int_{\mathcal{F}_t}\langle \lambda_1,\nu\rangle \gamma R^{\gamma-1}\,\mathrm{d}x
\\&=\int_0^t\!\!\int_{\mathcal{F}_\tau}
 (\gamma-1)\gamma R^{\gamma-2}\,\partial_\tau R\,\langle \lambda_1,\nu\rangle
 +(\gamma-1)\gamma R^{\gamma-2}\Big\langle \sqrt{\lambda_1}\lambda',\nu\Big\rangle\!\cdot\nabla R\,\mathrm{d}x\,\mathrm{d}\tau+\int_{\mathcal{F}_0}\gamma R_0^\gamma\,\mathrm{d}x.
\end{aligned}
\end{equation}

For the rigid part, taking \(\phi|_{\mathcal{B}^T}=\tilde u_B\) in \eqref{eq:MassCRigidP} gives
\begin{equation}\label{eq:RigidMomentum}
\begin{aligned}
    \frac{1}{2}\int_{\mathcal{B}_t}\rho_B|\tilde u_B|^2\,\mathrm{d}x
=\int_0^t\!\!\int_{\mathcal{B}_\tau}\rho_B\,\tilde u_B\cdot \partial_\tau \tilde u_B
+\rho_B u_B\cdot\big(\nabla\tilde u_B\,\tilde u_B\big)\,\mathrm{d}x\,\mathrm{d}\tau
+\frac{1}{2}\int_{\mathcal{B}_0}\rho_{B0}|u_{B0}|^2\,\mathrm{d}x.
\end{aligned}
\end{equation}

Plugging \eqref{testby U}--\eqref{eq:RigidMomentum} into \eqref{ExpandedRelativeEnergy}, and using the energy inequality \eqref{eq:EnergyMVS} for the dissipative Young measure-valued solution, we arrive at
\begin{equation}\label{SimplifiedErel}
\begin{aligned}
E_{\mathrm{rel}}(t)
&\le \int_{\mathcal{F}_t}R^\gamma\,\mathrm{d}x
+\int_0^t\!\!\int_{\mathcal{F}_\tau} U\cdot \partial_\tau U\,\langle \lambda_1,\nu\rangle
+\Big\langle \sqrt{\lambda_1}\lambda',\nu\Big\rangle\!\cdot(\nabla U\,U)\,\mathrm{d}x\,\mathrm{d}\tau\\
&\quad+\int_0^t\!\!\int_{\mathcal{B}_\tau}\rho_B(\tilde u_B-u_B)\cdot \partial_\tau \tilde u_B
+\rho_B u_B\cdot(\nabla\tilde u_B\,\tilde u_B)\,\mathrm{d}x\,\mathrm{d}\tau\\
&\quad-\int_0^t\!\!\int_{\mathcal{F}_\tau}\Big\langle \sqrt{\lambda_1}\lambda',\nu\Big\rangle\!\cdot \partial_\tau U
+\Big\langle \lambda'\!\otimes\!\lambda'+\lambda_1^\gamma\mathbb{I},\nu\Big\rangle\!:\!\nabla U\,\mathrm{d}x\,\mathrm{d}\tau
-\int_0^t\!\!\int_{\overline{\mathcal{F}_\tau}}\nabla U\,\mathrm{d}\nu^M\,\mathrm{d}\tau\\
&\quad-\gamma\int_0^t\!\!\int_{\mathcal{F}_\tau} R^{\gamma-2}\,\partial_\tau R\,\langle \lambda_1,\nu\rangle
+\gamma R^{\gamma-2}\Big\langle \sqrt{\lambda_1}\lambda',\nu\Big\rangle\!\cdot\nabla R\,\mathrm{d}x\,\mathrm{d}\tau
-\int_{\mathcal{F}_0}R_0^\gamma\,\mathrm{d}x-\mathcal{D}(t).
\end{aligned}
\end{equation}
Together with \eqref{ersatzpoincare}, this yields
\begin{equation}\label{eq:JL1}
\begin{aligned}
&E_{\mathrm{rel}}(t)+\mathcal{D}(t)
\\&\le \mathcal{O}(1)\int_0^t\mathcal{D}(\tau)\,\mathrm{d}\tau
+\underbrace{\int_{\mathcal{F}_t}R^\gamma\,\mathrm{d}x - \int_{\mathcal{F}_0}R_0^\gamma\,\mathrm{d}x
-\gamma\int_0^t\!\!\int_{\mathcal{F}_\tau} R^{\gamma-2}\,\partial_\tau R\,\langle \lambda_1,\nu\rangle\,\mathrm{d}x\,\mathrm{d}\tau}_{\mathcal{I}_0}\\
&\quad+\underbrace{\int_0^t\!\!\int_{\mathcal{F}_\tau}
  U\cdot \partial_\tau U\,\langle \lambda_1,\nu\rangle
 +\Big\langle \sqrt{\lambda_1}\lambda',\nu\Big\rangle\!\cdot(\nabla U\,U)
 -\Big\langle \sqrt{\lambda_1}\lambda',\nu\Big\rangle\!\cdot \partial_\tau U
 -\Big\langle \lambda'\!\otimes\!\lambda',\nu\Big\rangle\!:\!\nabla U\,\mathrm{d}x\,\mathrm{d}\tau}_{\mathcal{I}_1}\\
&\quad\underbrace{-\int_0^t\!\!\int_{\mathcal{F}_\tau}\langle \lambda_1^\gamma,\nu\rangle\,\nabla\!\cdot U
+\gamma R^{\gamma-2}\Big\langle \sqrt{\lambda_1}\lambda',\nu\Big\rangle\!\cdot\nabla R\,\mathrm{d}x\,\mathrm{d}\tau}_{\mathcal{I}_2}\\
&\quad+\underbrace{\int_0^t\!\!\int_{\mathcal{B}_\tau}
  \rho_B(\tilde u_B-u_B)\cdot \partial_\tau \tilde u_B
 +\rho_B u_B\cdot(\nabla\tilde u_B\,\tilde u_B)\,\mathrm{d}x\,\mathrm{d}\tau}_{\mathcal{I}_3}.
\end{aligned}
\end{equation}

To clarify the cancellations in \(\sum_{i=0}^3 \mathcal{I}_i\), we rewrite each term.

By the Reynolds transport theorem (see Appendix),
\begin{equation}\label{eq:I_0}
\begin{aligned}
\mathcal{I}_0
&=\int_{\mathcal{F}_t}R^\gamma\,\mathrm{d}x - \int_{\mathcal{F}_0}R_0^\gamma\,\mathrm{d}x
  -\gamma\int_0^t\!\!\int_{\mathcal{F}_\tau} R^{\gamma-2}\,\partial_\tau R\,\langle \lambda_1,\nu\rangle\,\mathrm{d}x\,\mathrm{d}\tau\\
&=\underbrace{\int_0^t\!\!\int_{\mathcal{F}_\tau}\gamma R^{\gamma-2}\,\partial_\tau R\,(R-\langle\lambda_1,\nu\rangle)\,\mathrm{d}x\,\mathrm{d}\tau}_{\mathcal{I}_{01}}
  +\int_0^t\!\!\int_{\partial\mathcal{F}_\tau} R^\gamma\,u_B\cdot n\,\mathrm{d}S\,\mathrm{d}\tau.
\end{aligned}
\end{equation}

Using the second line in \eqref{eq:transormedStrongSolution}, direct computations give
\begin{equation}\label{eq:I_1}
\begin{aligned}
\mathcal{I}_1
&=\int_0^t\!\!\int_{\mathcal{F}_\tau}
  \big\langle \lambda_1U-\sqrt{\lambda_1}\lambda',\nu\big\rangle\!\cdot \partial_\tau U
 +\big\langle U\otimes(\lambda_1U-\sqrt{\lambda_1}\lambda'),\nu\big\rangle\!:\!\nabla U\,\mathrm{d}x\,\mathrm{d}\tau\\
&\quad+\int_0^t\!\!\int_{\mathcal{F}_\tau}
  \big\langle (\lambda'-\sqrt{\lambda_1}U)\otimes(\sqrt{\lambda_1}U-\lambda'),\nu\big\rangle\!:\!\nabla U\,\mathrm{d}x\,\mathrm{d}\tau\\
&=\int_0^t\!\!\int_{\mathcal{F}_\tau}
  \big\langle (\lambda'-\sqrt{\lambda_1}U)\otimes(\sqrt{\lambda_1}U-\lambda'),\nu\big\rangle\!:\!\nabla U\,\mathrm{d}x\,\mathrm{d}\tau\\
&\quad+\int_0^t\!\!\int_{\mathcal{F}_\tau}
  \frac{\mathcal{G}-\mathcal{H}\,U}{R}\cdot \big\langle \lambda_1U-\sqrt{\lambda_1}\lambda',\nu\big\rangle\,\mathrm{d}x\,\mathrm{d}\tau
 -\int_0^t\!\!\int_{\mathcal{F}_\tau}
  \gamma R^{\gamma-2}\,\nabla R\cdot \big\langle \lambda_1U-\sqrt{\lambda_1}\lambda',\nu\big\rangle\,\mathrm{d}x\,\mathrm{d}\tau.
\end{aligned}
\end{equation}
Here $\mathcal{G}:=(\mathcal{G}^1, \mathcal{G}^2, \mathcal{G}^3)^\top$ with $\mathcal{G}^l, l=1,2,3$ given in \eqref{eq:transormedStrongSolution}.

Moreover,
\begin{equation}\label{eq:I_2}
\begin{aligned}
\mathcal{I}_2
&=-\int_0^t\!\!\int_{\mathcal{F}_\tau}\langle \lambda_1^\gamma,\nu\rangle\,\nabla\!\cdot U\,\mathrm{d}x\,\mathrm{d}\tau
   +\gamma\int_0^t\!\!\int_{\mathcal{F}_\tau} R^{\gamma-2}\Big\langle \sqrt{\lambda_1}\lambda',\nu\Big\rangle\!\cdot\nabla R\,\mathrm{d}x\,\mathrm{d}\tau\\
&=\int_0^t\!\!\int_{\mathcal{F}_\tau}\big(R^\gamma-\langle \lambda_1^\gamma,\nu\rangle\big)\,\nabla\!\cdot U\,\mathrm{d}x\,\mathrm{d}\tau
  -\int_0^t\!\!\int_{\partial\mathcal{F}_\tau}R^\gamma\,U\cdot n\,\mathrm{d}S\,\mathrm{d}\tau\\
&\quad+\underbrace{\int_0^t\!\!\int_{\mathcal{F}_\tau}\gamma R^{\gamma-2}\,\big\langle RU-\sqrt{\lambda_1}\lambda',\nu\big\rangle\!\cdot\nabla R\,\mathrm{d}x\,\mathrm{d}\tau}_{\mathcal{I}_{21}}.
\end{aligned}
\end{equation}

Since \(\tilde u_B(x)=\tilde V(t)+\tilde\omega(t)\times(x-X_1(t))\), the matrix \(\nabla\tilde u_B\) is skew-symmetric, hence
\begin{equation*}
    u_B\cdot(\nabla\tilde u_B\,u_B)=0, \qquad  (\tilde u_B-u_B)\cdot\big(\nabla\tilde u_B(\tilde u_B-u_B))=0.
\end{equation*}
 Therefore,
\begin{equation}\label{eq:I_3}
\begin{aligned}
\mathcal{I}_3
=&\int_0^t\!\!\int_{\mathcal{B}_\tau}\rho_B(\tilde u_B-u_B)\cdot \partial_\tau \tilde u_B
 +\rho_B u_B\cdot(\nabla\tilde u_B\,\tilde u_B)\,\mathrm{d}x\,\mathrm{d}\tau\\
=&\int_0^t\!\!\int_{\mathcal{B}_\tau}\rho_B\,(\partial_\tau \tilde u_B + u_B\cdot\nabla\tilde u_B)\cdot(\tilde u_B-u_B)\,\mathrm{d}x\,\mathrm{d}\tau\\
=&\int_0^t\!\!\int_{\mathcal{B}_\tau}\rho_B\,(\partial_\tau \tilde u_B + \tilde u_B\cdot\nabla\tilde u_B)\cdot(\tilde u_B-u_B)\,\mathrm{d}x\,\mathrm{d}\tau.
\end{aligned}
\end{equation}

Using the first equation in \eqref{eq:transormedStrongSolution},
\begin{equation}\label{eq:I_02}
\begin{aligned}
\mathcal{I}_{01}+\mathcal{I}_{21}
=&\int_0^t\!\!\int_{\mathcal{F}_\tau}\gamma R^{\gamma-2}(R-\langle\lambda_1,\nu\rangle)\,\mathcal{H}\,+\gamma R^{\gamma-2}\nabla R\cdot \langle\lambda_1U-\sqrt{\lambda_1}\lambda',\nu\rangle\,\mathrm{d}x\,\mathrm{d}\tau
 \\&+\int_0^t\!\!\int_{\mathcal{F}_\tau}\gamma R^{\gamma-1}\,\nabla\!\cdot U\,(\langle\lambda_1,\nu\rangle-R)\,\mathrm{d}x\,\mathrm{d}\tau.
\end{aligned}
\end{equation}

Substituting \eqref{eq:I_0}–\eqref{eq:I_02} into\eqref{eq:JL1}, we get
\begin{equation}\label{eq:RelEnergyfinalstep1}
\begin{aligned}
E_{\mathrm{rel}}(t)+\mathcal{D}(t)
&\le \int_0^t\!\!\int_{\mathcal{F}_\tau}\gamma R^{\gamma-2}(R-\langle\lambda_1,\nu\rangle)\,\mathcal{H}\,\mathrm{d}x\,\mathrm{d}\tau
  +\int_0^t\!\!\int_{\mathcal{F}_\tau}\gamma R^{\gamma-1}\,\nabla\!\cdot U\,(\langle\lambda_1,\nu\rangle-R)\,\mathrm{d}x\,\mathrm{d}\tau\\
&\quad+\int_0^t\!\!\int_{\mathcal{F}_\tau}\big\langle R^\gamma-\lambda_1^\gamma,\nu\big\rangle\,\nabla\!\cdot U\,\mathrm{d}x\,\mathrm{d}\tau
  +\int_0^t\!\!\int_{\mathcal{F}_\tau}\frac{\mathcal{G}-\mathcal{H}\,U}{R}\cdot\big\langle \lambda_1U-\sqrt{\lambda_1}\lambda',\nu\big\rangle\,\mathrm{d}x\,\mathrm{d}\tau\\
&\quad+\int_0^t\!\!\int_{\mathcal{F}_\tau}\big\langle (\lambda'-\sqrt{\lambda_1}U)\otimes(\sqrt{\lambda_1}U-\lambda'),\nu\big\rangle\!:\!\nabla U\,\mathrm{d}x\,\mathrm{d}\tau\\
&\quad+\int_0^t\!\!\int_{\mathcal{B}_\tau}\rho_B\,(\partial_\tau \tilde u_B + \tilde u_B\cdot\nabla\tilde u_B)\cdot(\tilde u_B-u_B)\,\mathrm{d}x\,\mathrm{d}\tau
+\mathcal{O}(1)\int_0^t\mathcal{D}(\tau)\,\mathrm{d}\tau.
\end{aligned}
\end{equation}

Collecting the 2nd, 3rd and 5th terms on the right-hand side of \eqref{eq:RelEnergyfinalstep1}, and using the definition \eqref{eq:defofRE}, we finally obtain
\begin{equation}\label{eq:simplifiedErelEst}
\begin{aligned}
E_{\mathrm{rel}}(t)+\mathcal{D}(t)
&\le \mathcal{O}(1)\int_0^t\big(E_{\mathrm{rel}}(\tau)+\mathcal{D}(\tau)\big)\,\mathrm{d}\tau\\
&\quad+\int_0^t\!\!\int_{\mathcal{B}_\tau}\rho_B\,(\partial_\tau \tilde u_B + \tilde u_B\cdot\nabla\tilde u_B)\cdot(\tilde u_B-u_B)\,\mathrm{d}x\,\mathrm{d}\tau\\
&\quad+\int_0^t\!\!\int_{\mathcal{F}_\tau}\gamma R^{\gamma-2}(R-\langle\lambda_1,\nu\rangle)\,\mathcal{H}\,\mathrm{d}x\,\mathrm{d}\tau
\\&\quad+\int_0^t\!\!\int_{\mathcal{F}_\tau}\frac{\mathcal{G}-\mathcal{H}\,U}{R}\cdot\big\langle \lambda_1U-\sqrt{\lambda_1}\lambda',\nu\big\rangle\,\mathrm{d}x\,\mathrm{d}\tau.
\end{aligned}
\end{equation}

Noting the last two equations in \eqref{eq:transormedStrongSolution} and following the procedure used to handle \cite[term (4.24)]{kreml2020weak}, we conclude that
\begin{equation}\label{eq:lastterm3}
\int_0^t\!\!\int_{\mathcal{B}_\tau}
\rho_B\,(\partial_\tau \tilde{u}_B + \tilde{u}_B\cdot\nabla\tilde{u}_B)\cdot(\tilde{u}_B - u_B)\,\mathrm{d}x\,\mathrm{d}\tau
\;\le\; \mathcal{O}(1)\int_0^t E_{\mathrm{rel}}(\tau)\,\mathrm{d}\tau.
\end{equation}

To estimate the last two terms on the right-hand side of \eqref{eq:simplifiedErelEst}, and following \cite[(4.27)–(4.36)]{kreml2020weak}, we split the dummy variable \(\lambda_1\) into
\[
\mathbb{K}_1:=\bigl\{\lambda_1:\lambda_1\in(\tfrac{R_-}{2},2R_+)\bigr\},\quad
\mathbb{K}_2:=\bigl\{\lambda_1:\lambda_1\in[0,\tfrac{R_-}{2}]\bigr\},\quad
\mathbb{K}_3:=\bigl\{\lambda_1:\lambda_1\in[2R_+,+\infty)\bigr\},
\]
where \(0<R_-:=\inf R\) and \(R_+:=\sup R<+\infty\).
Let the characteristic functions over \(\mathbb{K}_i\) be
\[
\zeta_i(\lambda_1):=\chi_{\mathbb{K}_i}(\lambda_1),\qquad i=1,2,3.
\]

We first record two elementary estimates, see for instance~\cite{FJN}.

\begin{lemma}\label{lem:Dummyvariableestmates}
For the ranges above and every \(\gamma>1\),
\[
(\lambda_1-R)^2\,\zeta_1
\;\le\; \mathcal{O}(1)\Big(\frac{\lambda_1^\gamma}{\gamma-1}-\frac{\gamma R^{\gamma-1}\lambda_1}{\gamma-1}+R^\gamma\Big)\zeta_1,
\]
and for \(k=2,3\),
\[
(1+\lambda_1^\gamma)\,\zeta_k
\;\le\; \mathcal{O}(1)\Big(\frac{\lambda_1^\gamma}{\gamma-1}-\frac{\gamma R^{\gamma-1}\lambda_1}{\gamma-1}+R^\gamma\Big)\zeta_k.
\]
\end{lemma}

We split the third term on the right-hand side of \eqref{eq:simplifiedErelEst} according to the above partition:
\begin{equation}\label{eq:I1-4}
    \begin{aligned}
        \Big|\int_{\mathcal{F}_t}\gamma R^{\gamma-2}(R-\langle\lambda_1,\nu\rangle)\mathcal{H}\mathrm{d}x\mathrm{d}\tau\Big|\leq \mathcal{O}(1)\sum_{i=1}^3\underbrace{\int_{\mathcal{F}_t}\big|\langle\qnt{R-\lambda_1}\zeta_i,\nu\rangle\big|\big|\mathcal{H}\big|\mathrm{d}x}_{\mathcal{I}'_i}.
    \end{aligned}
\end{equation}

By Lemma \ref{lem:Dummyvariableestmates}, we bound \(\mathcal{I}'_1\) and \(\mathcal{I}'_2\):
\begin{equation}\label{eq:II1}
    \begin{aligned}
        \mathcal{I}'_1&\leq \int_{\mathcal{F}_t}\langle(R-\lambda_1)^2\zeta_1,\nu\rangle +\mathcal{H}^2 \mathrm{d}x \\&\leq \mathcal{O}(1)\int_{\mathcal{F}_t}\langle\qnt{\frac{\lambda_1^\gamma}{\gamma-1}-\frac{\gamma R^{\gamma-1}\lambda_1}{\gamma-1}+R^\gamma},\nu\rangle\zeta_1\mathrm{d}x+\mathcal{O}(1)\|\mathcal{H}\|^2_{L^\infty(\mathcal{F}_t)}\\&\leq \mathcal{O}(1)E_{rel}(t)+\mathcal{O}(1)\|\mathcal{H}\|^2_{L^\infty(\mathcal{F}_t)},
    \end{aligned}
\end{equation}
and 
\begin{equation}\label{eq:II2}
    \begin{aligned}
        \mathcal{I}'_2&\leq \int_{\mathcal{F}_t}\langle(R-\lambda_1)^2\zeta_2,\nu\rangle +\mathcal{H}^2 \mathrm{d}x \leq \mathcal{O}(1)\int_{\mathcal{F}_t}\langle 1^2\zeta_2,\nu\rangle+\mathcal{H}^2\mathrm{d}x\\&\leq \mathcal{O}(1)\int_{\mathcal{F}_t}\langle\qnt{\frac{\lambda_1^\gamma}{\gamma-1}-\frac{\gamma R^{\gamma-1}\lambda_1}{\gamma-1}+R^\gamma},\nu\rangle\zeta_2\mathrm{d}x+\mathcal{O}(1)\|\mathcal{H}\|^2_{L^\infty(\mathcal{F}_t)}\\&\leq \mathcal{O}(1)E_{rel}(t)+\mathcal{O}(1)\|\mathcal{H}\|^2_{L^\infty(\mathcal{F}_t)}.
    \end{aligned}
\end{equation}
The term $\mathcal{I}'_3$ requires more careful treatment. Indeed, by Lemma \ref{lem:Dummyvariableestmates}, for $\gamma \in \left[1, 2\right]$,
\begin{equation}\label{eq:II3}
    \begin{aligned}
        \mathcal{I}'_3&\leq\mathcal{O}(1)\int_{\mathcal{F}_t}\langle|\lambda_1|\zeta_3,\nu\rangle\big|\mathcal{H}\big|\mathrm{d}x\leq\mathcal{O}(1)\|\langle|\lambda_1|\zeta_3,\nu\rangle\|_{L^{\gamma}(\mathcal{F}_t)}\|\mathcal{H}\|_{L^\infty(\mathcal{F}_t)}\\&\leq \mathcal{O}(1) \qnt{\int_{\mathcal{F}_t}\langle\lambda^\gamma_1\zeta_3,\nu\rangle\mathrm{d}x }^{\frac{2}{\gamma}}+\mathcal{O}(1)\|\mathcal{H}\|^2_{L^\infty(\mathcal{F}_t)}\\&\leq \mathcal{O}(1)E_{rel}^{\frac{2}{\gamma}}+\mathcal{O}(1)\|\mathcal{H}\|^2_{L^\infty(\mathcal{F}_t)};
    \end{aligned}
\end{equation}
for $\gamma\in (2,+\infty)$, 
\begin{equation}\label{eq:II4}
    \begin{aligned}
       \mathcal{I}'_3&\leq\mathcal{O}(1)\int_{\mathcal{F}_t}\langle|\lambda_1|\zeta_3,\nu\rangle\big|\mathcal{H}\big|\mathrm{d}x\leq \mathcal{O}(1)\int_{\mathcal{F}_t}\langle|\lambda_1|^{\frac{\gamma}{2}}\zeta_3,\nu\rangle\big|\mathcal{H}\big|\mathrm{d}x\\&\leq \mathcal{O}(1)\int_{\mathcal{F}_t}\langle|\lambda_1|^{\gamma}\zeta_3,\nu\rangle\mathrm{d}x+\mathcal{O}(1)\|\mathcal{H}\|^2_{L^\infty(\mathcal{F}_t)}\\&\leq \mathcal{O}(1)E_{rel}(t)+\mathcal{O}(1)\|\mathcal{H}\|^2_{L^\infty(\mathcal{F}_t)}.
    \end{aligned}
\end{equation}

By the specific formula of $\mathcal{H}$ and Lemma \ref{PropertiesonTransform},
\begin{equation}
    \|\mathcal{H}\|^2_{L^{\infty}(\mathcal{F}_t)}\leq\mathcal{O}(1)\|\partial_t\tilde{Z}_i(t,x)\|^2_{W^{1,\infty}(\mathcal{F}_t)}\leq\mathcal{O}(1)(\big|V-\tilde{V}\big|^2(t)+\big|\omega-\tilde{\omega}\big|^2(t)),
\end{equation}
which together with the definition of $E_{rel}(t)$ and the fact that 
\begin{equation}
   \begin{aligned}
       & \int_{\mathcal{B}_t}\rho_B|u_B-\tilde{u}_B|^2\mathrm{d}x\\&\ge\inf_{\mathcal{B}_t}\rho_B\int_{\mathcal{B}_t}\qnt{|V(t)-\tilde{V}(t)|^2+|(\omega(t)-\tilde{\omega}(t))\times(x-X_1(t))|^2}\mathrm{d}x\\&\ge c \qnt{|V(t)-\tilde{V}(t)|^2+|\omega(t)-\tilde{\omega}(t)|^2},
   \end{aligned}
\end{equation}
for some positive constants $c$, leads to
\begin{equation}\label{eq:H0}
    \|\mathcal{H}\|_{L^\infty(\mathcal{F}_t)}^2\leq \mathcal{O}(1)E_{rel}(t). 
\end{equation}
Combining \eqref{eq:II1}–\eqref{eq:H0} with \eqref{eq:I1-4}, we conclude that
\begin{equation}\label{eq:Finalstep1}
    \begin{aligned}
        &\big|\int_{\mathcal{F}_t}\gamma R^{\gamma-2}(R-\langle\lambda_1,\nu\rangle)\mathcal{H}\mathrm{d}x\mathrm{d}\tau\big|\leq \mathcal{O}(1)E_{rel}(t).
    \end{aligned}
\end{equation}

Next, we estimate the last term on the right-hand side of \eqref{eq:simplifiedErelEst}. First,
\begin{equation}\label{eq:Finalstep2}
   \begin{aligned}
       \big|\int_{\mathcal{F}_t} \frac{\mathcal{G}-\mathcal{H} U}{R}\cdot\langle\lambda_1U-\sqrt{\lambda_1}\lambda',\nu\rangle\mathrm{d}x\big|
       &=\big|\int_{\mathcal{F}_t} \frac{\mathcal{G}-\mathcal{H} U}{R}\cdot\langle(\lambda_1U-\sqrt{\lambda_1}\lambda')(\sum_{i=1}^3\zeta_i),\nu\rangle\mathrm{d}x\big|\\
       &\leq\sum_{i=1}^3\underbrace{\int_{\mathcal{F}_t}\big|\frac{\mathcal{G}-\mathcal{H} U}{R}\big|\langle\big|\lambda_1U-\sqrt{\lambda_1}\lambda'\big|\zeta_i,\nu\rangle\mathrm{d}x}_{\mathcal{I}^{''}_i}.
   \end{aligned} 
\end{equation}
By  Hölder's inequality and Jensen's inequality,
\begin{equation}\label{eq:III1}
    \begin{aligned}
        \mathcal{I}^{''}_1+\mathcal{I}^{''}_2&\leq \mathcal{O}(1)\|\frac{\mathcal{G}-\mathcal{H} U}{R}\|_{L^2(\mathcal{F}_t)}\qnt{\int_{\mathcal{F}_t}\langle\big|\sqrt{\lambda_1}U-\lambda'\big|^2,\nu\rangle\mathrm{d}x}^{\frac{1}{2}}\\&\leq \mathcal{O}(1)\|\frac{\mathcal{G}-\mathcal{H} U}{R}\|_{L^\infty(\mathcal{F}_t)}E_{rel}^{\frac{1}{2}}(t),
    \end{aligned}
\end{equation}
and 
\begin{equation}\label{eq:III2}
    \begin{aligned}
        \mathcal{I}^{''}_3&\leq \mathcal{O}(1)\|\frac{\mathcal{G}-\mathcal{H} U}{R}\|_{L^{\frac{2\gamma}{\gamma-1}}(\mathcal{F}_t)}\|\langle\big|\sqrt{\lambda_1}(\sqrt{\lambda_1}U-\lambda')\big|\zeta_3,\nu\rangle\|_{L^{\frac{2\gamma}{\gamma+1}}(\mathcal{F}_t)}\\&\leq \mathcal{O}(1)\|\frac{\mathcal{G}-\mathcal{H} U}{R}\|_{L^{\infty}(\mathcal{F}_t)}\|\langle\lambda_1^\gamma\zeta_3,\nu\rangle^{\frac{1}{2\gamma}}\langle\big|\sqrt{\lambda_1}U-\lambda'\big|^2,\nu\rangle^{\frac{1}{2}}\|_{L^{\frac{2\gamma}{\gamma+1}}(\mathcal{F}_t)}\\&\leq \mathcal{O}(1)\|\frac{\mathcal{G}-\mathcal{H} U}{R}\|_{L^{\infty}(\mathcal{F}_t)} \qnt{\int_{\mathcal{F}_t}\langle\lambda_1^\gamma\zeta_3,\nu\rangle\mathrm{d}x}^{\frac{1}{2\gamma}}\qnt{\int_{\mathcal{F}_t}\langle\big|\sqrt{\lambda_1}U-\lambda'\big|^2,\nu\rangle\mathrm{d}x}^{\frac{1}{2}}\\&\leq \mathcal{O}(1)\|\frac{\mathcal{G}-\mathcal{H} U}{R}\|_{L^{\infty}(\mathcal{F}_t)} E^{\frac{\gamma+1}{2\gamma}}_{rel}(t)
    \end{aligned}
\end{equation}
Thus, plugging \eqref{eq:III1} and \eqref{eq:III2} into \eqref{eq:Finalstep2}, we arrive at 
\begin{equation}\label{eq:Lastterm1}
    \begin{aligned}
        \big|\int_{\mathcal{F}_t} \frac{\mathcal{G}-\mathcal{H} U}{R}\cdot\langle\lambda_1U-\sqrt{\lambda_1}\lambda',\nu\rangle\mathrm{d}x\big|&\leq \mathcal{O}(1)\|\frac{\mathcal{G}-\mathcal{H} U}{R}\|_{L^{\infty}(\mathcal{F}_t)}\qnt{E^{\frac{\gamma+1}{2\gamma}}_{rel}(t)+E^{\frac{1}{2}}_{rel}(t)}\\&\leq \mathcal{O}(1) \|\frac{\mathcal{G}-\mathcal{H} U}{R}\|_{L^{\infty}(\mathcal{F}_t)}^2+\mathcal{O}(1)E_{rel}(t),
    \end{aligned}
\end{equation}
where Cauchy–Schwarz inequality and \(E_{\mathrm{rel}}\in L^\infty(0,T)\) are used in the last step.

According to the specific formulae for $\mathcal{H}$ and $\mathcal{G}$ in \eqref{eq:transormedStrongSolution}, noting the lower positive bound of the density $R$, applying Lemma \ref{PropertiesonTransform}, it holds that 
\begin{equation}
    \begin{aligned}
        \|\frac{\mathcal{G}-\mathcal{H}U}{R}\|^2_{L^{\infty}(\mathcal{F}_t)}&\leq \mathcal{O}(1)\qnt{|V-\tilde{V}|^2+|\omega-\tilde{\omega}|^2+\|V-\tilde{V}\|^2_{L^2(0,t)}+\|\omega-\tilde{\omega}\|^2_{L^2(0,t)}}\\&\leq \mathcal{O}(1)\qnt{E_{rel}(t)+\int_0^tE_{rel}(\tau)\mathrm{d}\tau}
    \end{aligned}
\end{equation}
which, together with \eqref{eq:Lastterm1}, gives that 
\begin{equation}\label{eq:Lastterm2}
    \begin{aligned}
        \big|\int_{\mathcal{F}_t} \frac{\mathcal{G}-\mathcal{H} U}{R}\cdot\langle\lambda_1U-\sqrt{\lambda_1}\lambda',\nu\rangle\mathrm{d}x\big|\leq \mathcal{O}(1)\qnt{E_{rel}(t)+\int_0^tE_{rel}(\tau)\mathrm{d}\tau}.
    \end{aligned}
\end{equation}

Finally, plugging \eqref{eq:Lastterm2}, \eqref{eq:Finalstep1} and \eqref{eq:lastterm3} into \eqref{eq:simplifiedErelEst}, we obtain
\begin{equation}
    \begin{aligned}
        E_{rel}(t)+\mathcal{D}(t)&\leq\mathcal{O}(1)\int_0^t E_{rel}(s)+\mathcal{D}(s)\mathrm{d}s+\mathcal{O}(1)\int_0^t\int_0^sE_{rel}(\tau)+\mathcal{D}(\tau)\mathrm{d}\tau \mathrm{d}s\\&\leq\mathcal{O}(1)\qnt{1+t}\int_0^t\qnt{E_{rel}(s)+\mathcal{D}(s)}\mathrm{d}s.
    \end{aligned} 
\end{equation}
Therefore, applying Gronwall's inequality, together with $E_{rel}(0)+\mathcal{D}(0)=0,$ we conclude that 
\begin{equation}\label{eq:FianlRealtiveest}
    E_{rel}(t)=\mathcal{D}(t)=0,\quad t\in (0,T),
\end{equation}
which together with  \eqref{eq:FianlRealtiveest} and Lemma \ref{PropertiesonTransform} yields that the maps $\tilde{Z}(t,\cdot)$ are the identity, i.e., 
\begin{equation}
    \tilde{Z}_i(t,x)=x.
\end{equation}
Together with \eqref{eq:TransformFluidp} and \eqref{eq:TransformedVelocity}, for $t\in(0,T)$, this leads to 
\begin{equation}\label{eq:JJ1}
    \begin{aligned}
       &\mathcal{B}_{2,t}=\mathcal{B}_{t},  \\
       &(\rho_{F2},u_{{F2}})\big|_{\mathcal{F}_t}=(R,U)\big|_{\mathcal{F}_t},\\
       &(\rho_{B2},u_{B2})\big|_{\mathcal{B}_t}=(\tilde{\rho}_B,\tilde{u}_B)\big|_{\mathcal{B}_t}.
    \end{aligned}
\end{equation}
From \eqref{eq:JJ1} and \eqref{eq:FianlRealtiveest} we get \(u_{B1}=u_{B2}\), and using the rigid mass equation (third line of \eqref{eq:model}) we infer
\begin{equation}\label{eq:JJ2}
\rho_{B1}=\rho_{B2}.
\end{equation}
Combining \eqref{eq:FianlRealtiveest}, \eqref{eq:JJ1}, and \eqref{eq:JJ2}, we conclude
\begin{equation}
    \begin{aligned}
        &\mathcal{B}_{1,t}=\mathcal{B}_{2,t}, t\in (0,T),\\
        &\nu=\chi_{\mathcal{F}_{1,t}}\delta_{(\rho_{F2},\sqrt{\rho_{F2}}u_{F2})}+\chi_{\mathcal{B}_{2,t}}\delta_{(\rho_{B2},\sqrt{\rho_{B2}}u_{B2})},\\
        &\mathcal{D}(t)=0,\quad t\in(0,T).
    \end{aligned}
\end{equation}
This completes the proof of Theorem~\ref{thm:Weak-strongUniqueness}.

 \medskip
 
 \section*{Acknowledgments}
 Qianfeng Li was partially supported by the Sino-German (CSC-DAAD) Postdoc Scholarship Program, 2023 (No.~57678375). Emil Wiedemann acknowledges support from DFG Priority Programme 2410 CoScaRa, project number 525716336.

 The authors would like to thank \v{S}\'arka Ne\v{c}asov\'a for valuable discussions on fluid-structure interaction.

 \section*{Declarations}
\noindent\textbf{Conflict of interest} 
On behalf of all authors, the corresponding author states that there is no conflict of interest.

\noindent\textbf{Data Availability} The paper does not use any data set.

\medskip
\appendix
\section{Reformulation of the solid's momentum equation}\label{App:WeakformofSolid}

\begin{lemma}[Reynolds Transport Theorem]\label{lem:reynolds}
Let $\mathcal{B}_t\subset\mathbb R^3$ be a smoothly time-dependent control volume with 
boundary $\partial\mathcal{B}_t$ moving with velocity field $v_B(x,t)$.
Let $\phi(x,t)$ be a scalar (or vector) field sufficiently smooth on $\mathcal{B}_t$.
Then for any $t$,
\begin{equation}\label{eq:reynolds}
\frac{\mathrm{d}}{\mathrm{d}t}\int_{\mathcal{B}_t}\phi(x,t)\,\mathrm{d}x
=
\int_{\mathcal{B}_t}\partial_t\phi(x,t)\,\mathrm{d}x
+\int_{\partial\mathcal{B}_t}\phi(x,t)\,v_B(x,t)\cdot n(x,t)\,\mathrm{d}S,
\end{equation}
where $n(x,t)$ is the outward unit normal on $\partial\mathcal{B}_t$.
Equivalently, using the divergence theorem,
\begin{equation}\label{eq:reynolds-div}
\frac{\mathrm{d}}{\mathrm{d}t}\int_{\mathcal{B}_t}\phi\,\mathrm{d}x
=
\int_{\mathcal{B}_t}\Big(\partial_t\phi+\operatorname{div}(\phi\,v_B)\Big)\,\mathrm{d}x.
\end{equation}
\end{lemma}

\begin{theorem}\label{thm:ApendixA}[From Newton--Euler to the distributional momentum equation with a surface delta]

In the inviscid compressible fluid-body system, suppose the Newton--Euler balances hold for all $t>0$:
\begin{align}
M\,{V}'(t) &= \int_{\partial\mathcal{B}_t}  p_Fn \, \mathrm{d}S, \label{NE-lin}\\
\frac{d}{dt}\big(\mathcal J(t)\,\omega(t)\big)
&= \int_{\partial\mathcal{B}_t} r \times p_Fn \, \mathrm{d}S, \label{NE-ang}
\end{align}
Here and in the sequel $r=x-X(t)$ and $X(t)$ is the mass center. Then the following \emph{distributional} momentum equation for the rigid phase holds on $\mathbb R^3\times(0,T)$:
\begin{equation}
\partial_t\big(\chi_{\mathcal{B}_t}\rho_B\, v\big)+\nabla\!\cdot\!\big(\chi_{\mathcal{B}_t}\rho_B\, v\otimes v\big)
= \big( p_Fn\big)\,\delta_{\partial\mathcal{B}_t}
\label{PDE-delta}
\end{equation}
for test functions in $C_c^\infty((0,T);\mathcal{R}(\Omega))$ where  
\begin{equation}
    \mathcal{R}(\Omega):=\{\xi:\Omega\to\mathbb{R}^3 \big| \xi=\alpha+\eta\times(\cdot-X), \alpha,\eta,X\in \mathbb{R}^3 \},
\end{equation}
$\chi_{\mathcal{B}_t}$ is the indicator of $\mathcal B_t$, and $\delta_{\partial\mathcal{B}_t}$ is the surface delta distribution
characterized by $\int_{\mathbb R^3}\phi\,\delta_{\partial\mathcal{B}_t}\,\mathrm{d}x=\int_{\partial\mathcal{B}_t}\phi\,\mathrm{d}S$.
\end{theorem}

\begin{proof}
We show that \eqref{PDE-delta} holds in the sense of distributions by verifying its space-time weak
form. Let $w\in C_c^\infty([0,T];\mathcal{R}(\Omega))$ be an arbitrary test function with
$w(\cdot,T)=0$. We need to prove
\begin{equation}
-\!\!\int_0^T\!\!\int_{\mathbb R^3} \chi_{\mathcal{B}_t}\rho_B\,v\cdot\partial_tw\,\mathrm{d}x\,\mathrm{d}t
-\!\!\int_0^T\!\!\int_{\mathbb R^3} \chi_{\mathcal{B}_t}\rho_B\,v\otimes v:\nabla w\,\mathrm{d}x\,\mathrm{d}t
=\int_0^T\!\!\int_{\partial\mathcal{B}_t} p_Fn\cdot w\,\mathrm{d}S\,\mathrm{d}t.
\label{weak}
\end{equation}
Because $\chi_{\mathcal{B}_t}$ localizes the integrals to $\mathcal B_t$, the left-hand side equals
\begin{equation}
-\!\!\int_0^T\!\!\int_{\mathcal B_t} \rho_B\,v\cdot\partial_t w\,\mathrm{d}x\,\mathrm{d}t
-\!\!\int_0^T\!\!\int_{\mathcal B_t} \rho_B\, v\otimes v:\nabla w\,\mathrm{d}x\,\mathrm{d}t.
\label{LHS}
\end{equation}

\noindent\textbf{Step 1: A kinematic transport identity.}
On the moving domain $\mathcal B_t$ with boundary velocity $v$, Reynolds' transport theorem
and the product rule yield, for every smooth $w$,
\begin{equation}
\frac{\mathrm{d}}{\mathrm{d}t}\int_{\mathcal B_t}\rho_B\,v\cdot w\,\mathrm{d}x
=\int_{\mathcal B_t}\rho_B\, v\cdot\partial_t w\,\mathrm{d}x
+\int_{\mathcal B_t}\rho_B\, v\otimes v:\nabla w\,\mathrm{d}x
+\int_{\mathcal B_t}\rho_B\, a\cdot w\,\mathrm{d}x,
\label{K0}
\end{equation}
where $ a:=\partial_t v+(v\!\cdot\!\nabla) v$ is the Eulerian acceleration.
Identity \eqref{K0} is purely kinematic (it uses mass conservation in the rigid body:
$\partial_t\rho_B+\operatorname{div}(\rho_B v)=0$).

\noindent\textbf{Step 2: Pairing the acceleration with the test function.}
For an arbitrary $w\in C_c^\infty([0,T];\mathcal{R}(\Omega))$, there exists $\alpha(t), \eta(t)$ such that 
\begin{equation}\label{proj}
   w(x,t)=\alpha(t)+\eta(t)\times r,\quad r=x-X(t). 
\end{equation}

Then the following pairing identity holds:
\begin{equation}
\int_{\mathcal B_t}\rho_B\,a\cdot w\,\mathrm{d}x
= M\,{V'}\cdot\alpha + \frac{\mathrm{d}}{\mathrm{d}t}\big(\mathcal J\omega\big)\cdot\eta.
\label{pair}
\end{equation}
Indeed, using the rigid kinematics
$a={V'}+{\omega'}\times r+\omega\times(\omega\times r)$,
the first term integrates to $M{V'}\cdot \alpha$ by \eqref{proj}, and the remaining
terms give, by the scalar triple-product identity and the identity
\begin{equation}
\frac{\mathrm{d}}{\mathrm{d}t}\big(\mathcal J\omega\big)
=\int_{\mathcal B_t}\rho_B\,\big(r\times a\big)\,\mathrm{d}x,
\label{Ldot}
\end{equation}
the second contribution in \eqref{pair}. Equality \eqref{Ldot} is obtained by differentiating the
angular momentum $\int_{\mathcal B_t}\rho_B\,( r\times v)\,\mathrm{d}x=\mathcal J\omega$
with Reynolds' theorem and using $\int_{\mathcal B_t}\rho_B\,(V\times v)\,\mathrm{d}x
= 0$.

\noindent\textbf{Step 3: Insert \eqref{pair} into \eqref{K0} and integrate in time.}
Combining \eqref{K0} and \eqref{pair} and rearranging gives
\begin{equation}
-\!\!\int_{\mathcal B_t}\rho_B\, v\cdot\partial_t w\,\mathrm{d}x
-\!\!\int_{\mathcal B_t}\rho_B\, v\otimes v:\nabla w\,\mathrm{d}x
= M\,{V'}\cdot \alpha + \frac{\mathrm{d}}{\mathrm{d}t}\big(\mathcal J\omega\big)\cdot\eta
-\frac{\mathrm{d}}{\mathrm{d}t}\int_{\mathcal B_t}\rho_B\,v\cdot w\,\mathrm{d}x.
\label{K}
\end{equation}
Integrating \eqref{K} over $(0,T)$ and using $w(\cdot,0)=w(\cdot,T) =0$ yields
%\begin{align}
%&-\!\!\int_0^T\!\!\int_{\mathcal B_t}\rho_B\, v\cdot\partial_t w\,\mathrm{d}x\,\mathrm{d}t
%-\!\!\int_0^T\!\!\int_{\mathcal B_t}\rho_B\, v\otimes v:\nabla w\,\mathrm{d}x\,\mathrm{d}t \nonumber\\
%&\qquad= \int_0^T \Big(M\,{V'}\cdot \alpha
%+\frac{\mathrm{d}}{\mathrm{d}t}\big(\mathcal J\omega\big)\cdot\eta\Big)\,\mathrm{d}t
%+\int_{\mathcal B_0}\rho_B\, v(\cdot,0)\cdot w(\cdot,0)\,\mathrm{d}x.
%\label{pre-weak}
%\end{align}
%The initial term can be absorbed in the usual way by requiring $w(\cdot,0)= 0$ (or by
%recording initial data separately). With $w(\cdot,0)= 0$, we obtain
\begin{equation}
-\!\!\int_0^T\!\!\int_{\mathcal B_t}\rho_B\,v\cdot\partial_tw\,\mathrm{d}x\,\mathrm{d}t
-\!\!\int_0^T\!\!\int_{\mathcal B_t}\rho_B\,v\otimes v:\nabla w\,\mathrm{d}x\,\mathrm{d}t
= \int_0^T \Big(M\,{V'}\cdot \alpha
+\frac{\mathrm{d}}{\mathrm{d}t}\big(\mathcal J\omega\big)\cdot\eta\Big)\,\mathrm{d}t.
\label{weak-left}
\end{equation}

\noindent\textbf{Step 4: Invoke Newton--Euler and identify the interface work.}
By the Newton--Euler relations \eqref{NE-lin}--\eqref{NE-ang},
\[
M\,{V'}\cdot \alpha
+\frac{\mathrm{d}}{\mathrm{d}t}\big(\mathcal J\omega\big)\cdot\eta
=\int_{\partial\mathcal{B}_t}p_Fn\cdot
\big(\alpha+\eta\times r\big)\,\mathrm{d}S=\int_{\partial\mathcal{B}_t} p_Fn\cdot w\,\mathrm{d}S.
\]
Therefore,
\[
\int_0^T \Big(M\,{V'}\cdot \alpha
+\frac{\mathrm{d}}{\mathrm{d}t}\big(\mathcal J\omega\big)\cdot\eta\Big)\,\mathrm{d}t
=\int_0^T\!\!\int_{\partial\mathcal{B}_t} p_Fn\cdot w\,\mathrm{d}S\,\mathrm{d}t.
\]
Substituting this into \eqref{weak-left} gives \eqref{weak}. Thus the distribution
identity \eqref{PDE-delta} follows by the definition of the surface delta:
\[
\int_0^T\!\!\int_{\mathbb R^3}\big(p_Fn\,\delta_{\partial\mathcal{B}_t}\big)\cdot w\,\mathrm{d}x\,\mathrm{d}t
=\int_0^T\!\!\int_{\partial\mathcal{B}_t} p_Fn\cdot w\,\mathrm{d}S\,\mathrm{d}t.
\]
This completes the proof.
\end{proof}

\section{Proof of Lemma \ref{lem:ConvergeBou1}}
\label{app:proof-lem-ConvergeBou1}

\begin{proof}[Proof of Lemma \ref{lem:ConvergeBou1}] 

The convergence \eqref{eq:ConvergeBou} follows directly from \cite[Lemma 5.4]{gerard2014existence}. It remains to verify \eqref{eq:ConvergeBoundary}.

For each $t\in[0,T],$ define
\[
a_t^\varepsilon := \mathbb{O}_t^\top\big(X^\varepsilon(t)-X(t)\big)\in\mathbb{R}^3,
\qquad
S_t^\varepsilon := \mathbb{O}_t^\top \mathbb{O}_t^\varepsilon \in SO(3).
\]
Since $\mathcal{B}_t=X(t)+\mathbb{O}_t\mathcal{B}_0$ and
$\mathcal{B}_t^\varepsilon=X^\varepsilon(t)+\mathbb{O}_t^\varepsilon\mathcal{B}_0$, we have
\[
\mathbb{O}_t^\top(\mathcal{B}_t-X(t))=\mathcal{B}_0,
\qquad
\mathbb{O}_t^\top(\mathcal{B}_t^\varepsilon-X(t))
=a_t^\varepsilon + S_t^\varepsilon \mathcal{B}_0.
\]
Using invariance of Lebesgue measure under rotations, it follows that
\begin{equation}\label{eq:pullback_sd}
\mathcal{L}^3(\mathcal{B}_t^\varepsilon\Delta\mathcal{B}_t)
=
\mathcal{L}^3\big(\mathcal{B}_0 \Delta (a_t^\varepsilon + S_t^\varepsilon\mathcal{B}_0)\big)
\qquad \text{for all } t\in[0,T].
\end{equation}

We introduce the functional
\[
\mathcal{E}(a,S)
:=\mathcal{L}^3\big(\mathcal{B}_0 \Delta (a+S\mathcal{B}_0)\big),
\qquad (a,S)\in\mathbb{R}^3\times SO(3).
\]
Then $\mathcal{E}$ is continuous, and
\begin{equation}\label{eq:zeroset_F}
\mathcal{E}(a,S)=0
\quad\Longleftrightarrow\quad
a=0 \ \text{and}\ S\in G.
\end{equation}
Indeed, $\mathcal{E}(a,S)=0$ means $\mathcal{B}_0=a+S\mathcal{B}_0$ (a.e.), which implies $a=0$
since $\mathcal{B}_0$ is bounded; then $S\mathcal{B}_0=\mathcal{B}_0$, i.e.\ $S\in G$.

Fix $r>0$. Since $SO(3)$ is compact and $G$ is a closed subgroup of $SO(3)$, $G$ is compact.
Consider the compact set
\[
K_r:=\big\{(a,S)\in\mathbb{R}^3\times SO(3): |a|\le M,\ |a|+\mathrm{dist}(S,G)\ge r\big\},
\]
where $M>0$ is chosen so that if $|a|\ge M$ then
$\mathcal{B}_0\cap(a+S\mathcal{B}_0)=\emptyset$ (hence $\mathcal{E}(a,S)=2\mathcal{L}^3(\mathcal{B}_0)$).
By continuity of $\mathcal{E}$ and \eqref{eq:zeroset_F}, $\mathcal{E}>0$ on $K_r$, hence
\[
\kappa(r):=\min_{(a,S)\in K_r}\mathcal{E}(a,S)>0.
\]
Consequently,
\begin{equation}\label{eq:coercive_F}
|a|+\mathrm{dist}(S,G)\ge r
\quad\Longrightarrow\quad
\mathcal{E}(a,S)\ge \kappa(r).
\end{equation}

Assume by contradiction that there exist $\varepsilon_k\to0$ and $r>0$ such that
\[
\sup_{t\in[0,T]}\Big(|a_t^{\varepsilon_k}|+\mathrm{dist}(S_t^{\varepsilon_k},G)\Big)\ge 2r
\quad\text{for all }k.
\]
Then choose $t_k\in[0,T]$ so that
$|a_{t_k}^{\varepsilon_k}|+\mathrm{dist}(S_{t_k}^{\varepsilon_k},G)\ge r$.
By \eqref{eq:coercive_F} and \eqref{eq:pullback_sd},
\[
\mathcal{L}^3(\mathcal{B}_{t_k}^{\varepsilon_k}\Delta\mathcal{B}_{t_k})
=\mathcal{E}(a_{t_k}^{\varepsilon_k},S_{t_k}^{\varepsilon_k})
\ge \kappa(r),
\]
which contradicts \eqref{eq:ConvergeBou}. Hence,
\begin{equation}\label{eq:dist_to_sym}
\lim_{\varepsilon\to0}\ \sup_{t\in[0,T]}
\Big(|a_t^\varepsilon|+\mathrm{dist}(S_t^\varepsilon,G)\Big)=0.
\end{equation}

Fix $t\in[0,T]$. Choose $Q\in G$ such that
$\|S_t^\varepsilon-Q\|=\mathrm{dist}(S_t^\varepsilon,G)$ (existence follows from compactness of $G$).
Then for any $x\in\partial\mathcal{B}_0$,
\[
y^\varepsilon(t,x)=X(t)+\mathbb{O}_t\big(a_t^\varepsilon+S_t^\varepsilon x\big),
\qquad
y(t,Qx)=X(t)+\mathbb{O}_t(Qx).
\]
Therefore, using that $\mathbb{O}_t$ is an isometry and $\partial\mathcal{B}_0$ is bounded,
\[
|y^\varepsilon(t,x)-y(t,Qx)|
\le |a_t^\varepsilon|+\|S_t^\varepsilon-Q\|\,|x|
\le |a_t^\varepsilon|+C_0\,\mathrm{dist}(S_t^\varepsilon,G).
\]
Taking $\sup_{x\in\partial \mathcal{B}_0}$, then $\inf_{Q\in G}$, then $\sup_{t\in[0,T]}$ and using \eqref{eq:dist_to_sym},
we obtain the first estimate in \eqref{eq:ConvergeBoundary}.

For normals, let $n_0(x)$ be the outward unit normal to $\partial\mathcal{B}_0$. Since
$\partial\mathcal{B}_0$ is regular and $Q\in SO(3)$ is an isometry, $n_0(Qx)=Q\,n_0(x)$.
Moreover,
\[
n(y^\varepsilon(t,x))=\mathbb{O}_t^\varepsilon n_0(x)=\mathbb{O}_t S_t^\varepsilon n_0(x),
\qquad
n(y(t,Qx))=\mathbb{O}_t n_0(Qx)=\mathbb{O}_t Q n_0(x).
\]
Hence
\[
|n(y^\varepsilon(t,x))-n(y(t,Qx))|
\le \|S_t^\varepsilon-Q\|
=\mathrm{dist}(S_t^\varepsilon,G).
\]
Taking $\sup_{x\in\partial \mathcal{B}_0}$, $\inf_{Q\in G}$, then $\sup_{t\in[0,T]}$, and and using \eqref{eq:dist_to_sym},we obtain the second estimate in
\eqref{eq:ConvergeBoundary}. The proof is complete.
\end{proof}

\bibliographystyle{plain}
\bibliography{CKWX1014}

\end{document}